\documentclass[12pt,reqno]{amsart}
\usepackage{amssymb,amsmath,dsfont}
\usepackage{pgfplots}
\pgfplotsset{compat=1.14}
\usepackage{enumitem}
\usepackage{calc}
\usepackage{mathrsfs}

\title[On the maximal dimension of an irreducible representation of $S_N$]{On the maximal dimension of an irreducible representation of the symmetric group}
\date{}

\author{Amol Aggarwal}
\author{Dor Elboim}

\usepackage[margin=1in]{geometry}
\usepackage[all]{xy}
\usepackage{amssymb}
\usepackage{amsfonts}
\usepackage{amsthm}
\usepackage{amsmath}
\usepackage{hyperref}
\usepackage{mathtools}
\usepackage{url}
\usepackage{xcolor}
\usepackage{dsfont}
\usepackage{pgfplots}
\usepackage{bbm}

\newtheorem{thm}{Theorem}
\newtheorem*{thm*}{Theorem}

\newtheorem{lem}{Lemma}[section]  
\newtheorem{prop}[lem]{Proposition}
\newtheorem{cor}[lem]{Corollary}

\theoremstyle{remark}

\DeclareMathOperator{\len}{len}

\numberwithin{equation}{section}

\begin{document}

\begin{abstract}
    We prove that the maximal dimension $d_N$ of an irreducible representation of the symmetric group $S_N$ satisfies
        $$d_N=\sqrt{N!} \, e^{-(\mathfrak{d}+o(1))\sqrt{N} }, \quad N\to \infty,$$
    for some constant $\mathfrak{d}>0$. This answers a question raised by Vershik--Kerov in 1985.
\end{abstract}

\maketitle
\tableofcontents 

\section{Introduction}

Let $S_N$ denote the symmetric group on a set of $N \ge 1$ elements. The irreducible representations of $S_N$ are indexed by the set $\mathbb{Y}_N$ of partitions of size $N$. For any $\lambda \in \mathbb{Y}_N$ let $\dim \lambda$ denote the dimension of the irreducible representation of $S_N$ associated with $\lambda$. A natural question is how the maximal dimension 
\begin{flalign*} 
d_N = \max_{\lambda \in \mathbb{Y}_N} \dim \lambda,
\end{flalign*} 

\noindent of an irreducible representation of $S_N$ grows with $N$. Indeed, this dates back to the report of Bivins--Metropolis--Stein--Wells \cite{CSG}, which was among the earlier works to advocate and implement the use of computer calculations to better understand combinatorial phenomena. Ironically, for this specific question, numerical data was not always decisive. 

Since we have $\sum_{\lambda \in \mathbb{Y}_N} (\dim \lambda)^2 = |S_N| = N!$ and $|\mathbb{Y}_N| \le e^{(\pi \sqrt{2/3} + o(1)) \sqrt{N}}$, it follows that $d_N \ge e^{-(c_1+o(1)) \sqrt{N}} \sqrt{N!}$ for $c_1 = \pi / \sqrt{6}$. Based on a tabulation of $d_N$ for $N \le 75$, McKay~\cite{mckay1976largest} hypothesized the improved lower bound $d_N \ge N^{-1} \sqrt{N!}$; however, the counterexample $N = 81$ was found soon afterwards. Vershik--Kerov \cite{vershik1985asymptotic} later showed that this was false in a strong asymptotic sense, by proving that $d_N \le e^{-(c_2+o(1)) \sqrt{N}} \sqrt{N!}$, where $c_2 = 2(\pi-2) / \pi^2$. 

Together with the above mentioned lower bound, this led them to ask whether $d_N = e^{-(\mathfrak{d} + o(1)) \sqrt{N}} \sqrt{N!}$ for some constant $\mathfrak{d} > 0$, that is, whether the limit 
\begin{flalign}
\label{dnlimit} 
\displaystyle\lim_{N \rightarrow \infty} \displaystyle\frac{1}{\sqrt{N}} \log  \displaystyle\frac{d_N}{\sqrt{N!}}, 
\end{flalign} 

\noindent exists. In subsequent years, several computational experiments were performed to test this. The first ones of Kerov--Pass \cite{RMSG} led them to predict that \eqref{dnlimit} does exist, while more recent ones of Vershik--Pavlov \cite[Section 3]{vershik2009numerical} led them to suggest that it might not. Still, it seems to the broader opinion that the limit \eqref{dnlimit} should exist (and this is supported by later numerics of Duzhin--Vasilyev \cite{ABND}). See, for example, the survey \cite[Section 2.2]{SSA} of Pak (where this prediction is referred to as the Vershik--Kerov--Pass conjecture). 

In this paper we confirm the existence of this limit. 

\begin{thm}\label{thm:main}
    There exist constants $\mathfrak{d} > 0$ and $C>0$ such that, for any integer $N \ge 2$,  
    \begin{flalign}
    \label{dnlimit2} 
        \Big| \displaystyle\frac{1}{\sqrt{N}} \log \displaystyle\frac{d_N}{\sqrt{N!}}  + \mathfrak{d} \Big| \le \displaystyle\frac{C}{\log N}.
    \end{flalign}
    
\end{thm}

We do not have an explicit formula for the constant $\mathfrak{d}$ in Theorem \ref{thm:main}, nor do we have any reason to suspect that one should exist. The error $C(\log N)^{-1}$ in \eqref{dnlimit2} is not expected to be optimal; we would speculate that it should be closer to $CN^{-1/2}$. 

The existence of the limit \eqref{dnlimit} is closely related to the study of log gases. For any $\lambda = (\lambda_1, \lambda_2, \ldots , \lambda_{\ell}) \in \mathbb{Y}_N$, the hook length formula implies that 
\begin{flalign*}
\dim \lambda = N! \displaystyle\prod_{1 \le i < j \le \ell} (\lambda_i - \lambda_j - i +j) \cdot  \displaystyle\prod_{i=1}^{\ell} (\lambda_i - i + \ell)!^{-1}.
\end{flalign*} 

\noindent For $\beta>0$, let $\mathbb{P}_{\beta}^N$ denote the probability measure that gives $\lambda \in \mathbb{Y}_N$ weight proportional to $(\dim \lambda)^{\beta}$. Then $\mathbb{P}_{\beta}^N$ can be viewed as a discrete log gas, namely, as a random system of particles on $\mathbb{Z}$ interacting through a logarithmic potential. Theorem \ref{thm:main} is equivalent to the convergence of $N^{-1/2} \log (N!^{-1/2} \dim \lambda)$ to a constant, almost surely as $N$ tends to $\infty$, for $\lambda$ sampled under $\mathbb{P}_{\infty}^N$. Vershik \cite{vershik1989statistical} predicted that this also holds for any $\beta > 0$. 

The Plancherel measure for $S_N$ is $\mathbb{P}_2^N$, which was shown to be a determinantal point process by Borodin--Okounkov--Olshanski \cite{AMSG} and Johansson \cite{DOPM}. Facilitated by this, Bufetov \cite{CTFM} proved $N^{-1/2} \log (N!^{-1/2} \dim \lambda)$ has an almost sure limit, for $\lambda$ sampled under $\mathbb{P}_{\beta}^N$, when $\beta=2$; this was extended to certain other determinantal, Plancherel-type measures on partitions by Mkrtchyan \cite{EM}. However, the determinantal structure for $\mathbb{P}_{\beta}^N$ is lost when $\beta \ne 2$. 

Away from this determinantal point $\beta=2$, more is known for continuum log gases, namely, when the $(\lambda_i)$ are not restricted to be integers. In this continuous setting, the counterpart of the existence of the limit \eqref{dnlimit} is closely related to pinpointing the local spacings between the particles $(\lambda_i)$. This was addressed by Sandier--Serfaty \cite{LRE}, who showed that, as $\beta$ tends to $\infty$, the particles $(\lambda_i)$ locally crystallize around an evenly spaced lattice. Analogs of these results have seemingly not been developed for discrete log gases, such as $\mathbb{P}_{\beta}$, where certain differences arise (for example, this crystallization phenomenon cannot quite be valid as stated at points where the particle density $(\lambda_i)$ is irrational, as the $(\lambda_i)$ must be integers). 

We will comment further on the proof of Theorem \ref{thm:main} in Section \ref{Proof0} below. We find it plausible that this method might also be of use to study the predicted \cite{vershik1989statistical} almost sure limit of $N^{-1/2} \log (N!^{-1/2} \dim \lambda)$, for $\lambda$ sampled under $\mathbb{P}_{\beta}$ at $\beta \notin \{ 2, \infty \}$. However, we will not pursue this here.

\subsection*{Acknowledgements} 

Amol Aggarwal heartily thanks Greta Panova for valuable discussions on this question. The work of Amol Aggarwal was partially supported by Packard Fellowship for Science and Engineering.

\section{Preliminaries}

\subsection{Partitions and the hook length formula}
As mentioned above, irreducible representations of $S_N$ are in bijection with partitions of $N$ (see, e.g., \cite[Theorem~2.4.6]{sagan2001symmetric}).
Recall that a partition $\lambda$ of $N$ is a vector
$\lambda=(\lambda_1,\dots,\lambda_k)$ of positive integers with
$\lambda_1\ge \lambda_2\ge \cdots \ge \lambda_k$ such that
$\lambda_1+\cdots+\lambda_k=N$. The Young diagram of
$\lambda=(\lambda_1,\dots,\lambda_k)$ is a union of squares, called cells, in the plane given by
\begin{equation}
   Y(\lambda)
   :=
   \bigcup_{i=1}^k \bigcup_{j=1}^{\lambda_i}
   [j-1,j]\times[-i,-i+1].
\end{equation}
The $i$-th row of the diagram consists of $\lambda_i$ cells. See Figure~\ref{fig:1}.

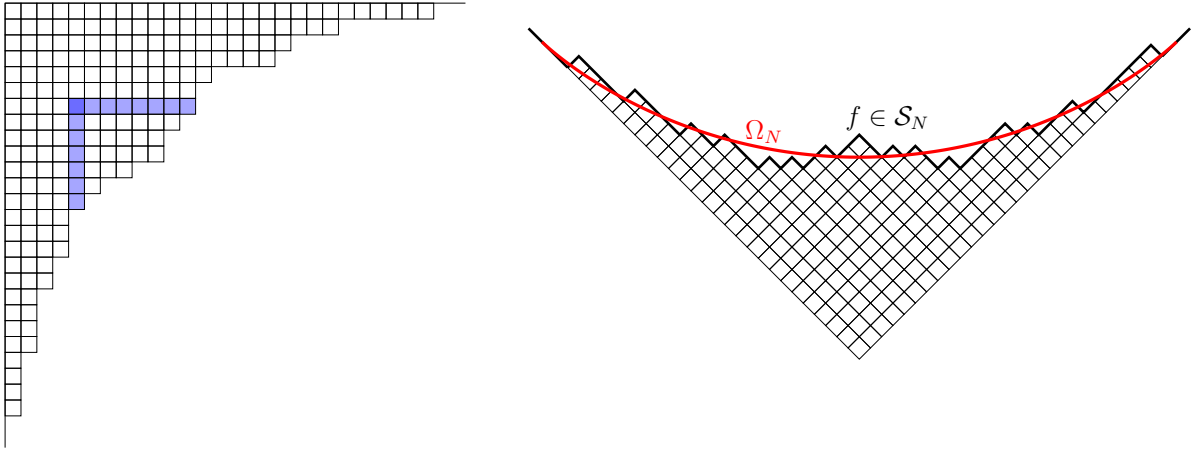
\begin{figure}[ht]
\centering


\def\PartList{27,21,18,17,13,12,12,11,10,10,8,6,5,4,4,4,3,3,2,2,2,2,1,1,1,1}

\pgfmathsetmacro{\cell}{0.21}
\pgfmathsetmacro{\rcell}{\cell/sqrt(2)}
\pgfmathsetmacro{\R}{2*sqrt(200)}   

\newcommand{\DrawYoungEnglish}{
    \foreach[count=\r from 0] \len in \PartList {
        \foreach \c in {1,...,\len} {
            \draw (\c-1,-\r) rectangle (\c,-\r-1);
        }
    }
}

\newcommand{\FillHook}{

    \foreach \c in {5,...,12} {
        \fill[blue,opacity=0.35] (\c-1,-6) rectangle (\c,-7);
    }

    \foreach \r in {7,...,13} {
        \fill[blue,opacity=0.35] (4,1-\r) rectangle (5,-\r);
    }
}

\newcommand{\DrawYoungRussian}{
    \foreach[count=\i from 1] \len in \PartList {
        \foreach \j in {1,...,\len} {
            \draw (\j-\i,\i+\j-2) --
                  (\j-\i+1,\i+\j-1) --
                  (\j-\i,\i+\j) --
                  (\j-\i-1,\i+\j-1) -- cycle;
        }
    }
}

\newcommand{\DrawRussianUpperBoundary}{
    \draw[line width=1pt]
    (-26,26) --
    (-25,27) --
    (-24,26) --
    (-23,25) --
    (-22,24) --
    (-21,23) --
    (-20,24) --
    (-19,23) --
    (-18,22) --
    (-17,21) --
    (-16,20) --
    (-15,21) --
    (-14,20) --
    (-13,19) --
    (-12,20) --
    (-11,19) --
    (-10,18) --
    (-9,17) --
    (-8,18) --
    (-7,17) --
    (-6,18) --
    (-5,17) --
    (-4,18) --
    (-3,19) --
    (-2,18) --
    (-1,19) --
    (0,20) --
    (1,19) --
    (2,18) --
    (3,19) --
    (4,18) --
    (5,19) --
    (6,18) --
    (7,17) --
    (8,18) --
    (9,17) --
    (10,18) --
    (11,19) --
    (12,20) --
    (13,21) --
    (14,20) --
    (15,21) --
    (16,20) --
    (17,21) --
    (18,22) --
    (19,23) --
    (20,22) --
    (21,23) --
    (22,24) --
    (23,25) --
    (24,26) --
    (25,27) --
    (26,28) --
    (27,27);
}

\begin{tikzpicture}[line width=0.35pt]

    \begin{scope}[x=\cell cm,y=\cell cm,shift={(0,26)}]

        \draw (0,-28) -- (0,0);
        \draw (0,0) -- (29,0);

        \FillHook
        \DrawYoungEnglish
    \end{scope}

    \begin{scope}[shift={(11.3cm,0.75cm)},x=\rcell cm,y=\rcell cm]

        \DrawYoungRussian

        \DrawRussianUpperBoundary

        \draw[line width=1pt] (-26,26) -- (-\R-1.2,\R+1.2);
        \draw[line width=1pt] (27,27) -- (\R+1.2,\R+1.2);

        \draw[red,very thick,samples=400,smooth,variable=\t,domain=-pi/2:pi/2]
            plot ({\R*sin(\t r)},
                  {(2*\R/pi)*(\t*sin(\t r)+cos(\t r))});

        \node[scale=0.85] at (2.5,21.5) {$f\in \mathcal S_N$};
        \node[scale=0.85] at (-8.5,20.3) {{\color{red}$\Omega_N$}};

    \end{scope}

\end{tikzpicture}
\caption{The left figure is the Young diagram of a partition of size 200 sampled from the Plancherel measure. The hook of cell $(7,5)$ appears in blue with the hook length being $14$.
The right figure shows the function $f\in \mathcal S _N$ corresponding to that partition. On top of it in red is the rescaled version of limiting function $\Omega $.}
\label{fig:1}
\end{figure}

We denote by $\dim \lambda $ the dimension of the irreducible representation corresponding to $\lambda $. The hook length formula (see, e.g., \cite[Theorem~3.10.2]{sagan2001symmetric}) states that 
\begin{equation}\label{eq:hl}
    \dim \lambda =\frac{N!}{\prod_{(i,j)} h_{i,j}},
\end{equation}
where the product is over the cells of the Young diagram of $\lambda $ and where $h_{i,j}$ is the hook length of cell $(i,j)$. The hook length $h_{i,j}$ is the number of cells in the hook with corner at $(i,j)$. Namely, this is the number of cells in the diagram to the right of $(i,j)$ plus the number of cells below $(i,j)$ plus one (the cell itself). See Figure~\ref{fig:1}.

Hence, finding the representation with maximal dimension corresponds to finding the partition for which the product of hooks $\prod _{(i,j)}h_{i,j}$ is minimal.

\subsection{Step functions}

We parametrize partitions using step functions. Let $\mathcal S$ be the set of continuous, piecewise linear functions $f:\mathbb R \to \mathbb R$ such that $f'(x)\in \{-1,1\}$ for all $x\notin \mathbb Z$ and such that $f(x)=|x|$ when $|x|$ is sufficiently large. This is the set of functions $f$ whose graph is the boundary of a rotated and rescaled Young diagram. See Figure~\ref{fig:1}. We let $\lambda (f)$ be the partition corresponding to $f$. For $f\in \mathcal S$ let
\begin{equation*}
    N(f):= \frac{1}{2}\int _{-\infty }^{\infty } (f(x)-|x|)dx.
\end{equation*}
This is the size of the partition $\lambda (f)$ or the number of cells in the Young diagram of $\lambda (f)$. We let
\begin{equation*}
    \mathcal S_N:=\{f\in \mathcal S: N(f)=N\}
\end{equation*}
be the functions corresponding to partitions of size $N$.

\subsection{The limit shape $\Omega$ }
Taking the logarithm of \eqref{eq:hl} turns the right hand side into a sum of logarithms of hook lengths. This sum is the Riemann sum of the continuous hook integral (see, e.g., \cite[Section~1.12]{romik2015surprising}). Hence, finding the maximal dimension of a partition leads to a continuous variational problem of minimizing the hook integral under an area constraint (see \cite[Section~1.13]{romik2015surprising}). The solution of the variational problem is the \emph{Logan--Shepp} \cite{VRT} and \emph{Vershik--Kerov} \cite{vershik1985asymptotic} limit shape, defined by
\begin{equation}
    \Omega (x):= \begin{cases}
    \frac{2}{\pi }\big( x\arcsin x +\sqrt{1-x^2} \big)\quad &|x|\le 1\\
    |x|\quad &|x|\ge 1    
    \end{cases}
\end{equation}
It follows that the step function corresponding to a partition with close to maximal dimension (including a partition sampled from the Plancherel measure) will be somewhat close to a rescaled version of $\Omega $. Precisely, to match the scale of functions $f\in \mathcal S _N$, we rescale $\Omega $ by  $\Omega _N(x)=\sqrt{4N}\Omega (x/\sqrt{4N})$. See Figure~\ref{fig:1} and Figure~\ref{fig:2}.

\subsection{Expansion around $\Omega $}

The starting point of our analysis is the following result due to \cite[Lemmas~1,2,3,4]{vershik1985asymptotic}. The result gives a formula for the dimension $\dim \lambda (f)$ in terms of the Sobolev norm $H^{1/2}$ of $f-\Omega _N$. To state the result, we will need some notations. For $f\in \mathcal S$, define
\begin{equation}
        \theta _N(f) :=\int _{-\infty } ^\infty \int _{-\infty } ^\infty  \Big( \frac{f(x)-f(y)}{x-y} -\frac{\Omega _N(x)-\Omega _N(y)}{x-y}\Big) ^2 dx\, dy
    \end{equation}
    and
        \begin{equation}
        \bar \theta _N(f) :=\int _{|x|\ge 2\sqrt{N}}  (f(x)-|x|)\,  {\rm arccosh} \Big(\frac{|x|}{2\sqrt{N}} \Big) dx,
    \end{equation}
    where ${\rm arccosh}(x)=\log (x+\sqrt{x^2-1})$. Finally define 
    \begin{equation}\label{eq:tilde1}
        \tilde \theta (f):= \sum _{(i,j)} \varphi (h_{i,j}), \quad \text{where} \quad \varphi (x):=\sum _{k=1}^{\infty } \frac{1}{k(k+1)(2k+1)x^{2k}},
    \end{equation}
and where the sum in the definition of $\tilde \theta $ is over the cells of the Young diagram of $\lambda (f)$.

\begin{prop}[Vershik-Kerov]\label{prop:VK}
    For any $N\ge 1$ and  $f\in \mathcal S _N$ we have that
    \begin{equation}\label{eq:VK}
        -\log \bigg( \frac{\dim ^2\lambda (f) }{(N!)^2(N/e)^{-N}} \bigg)= \log \bigg( \frac{\prod _{(i,j)}h_{i,j}^2}{(N/e)^{N}} \bigg)=\frac{1}{8}\theta _N(f) +\tilde \theta (f) +\bar \theta _N(f).
    \end{equation}
\end{prop}

\begin{proof}
    Let us explain how to derive this formula from \cite{vershik1985asymptotic}. In \cite[Equation~(2)]{vershik1985asymptotic} the authors consider the quantity $\mu _N(\lambda )=\dim ^2(\lambda )/N!$. Here we replace $N!$ with $(N!)^2(N/e)^{-N}$, which is almost the same asymptotically, and define $\mu _N'(\lambda ):=\frac{\dim ^2(\lambda )}{(N!)^2(N/e)^{-N}}$. This is done in order to obtain an exact equality in the statement of the theorem instead of an approximation. The error term $\epsilon _N$ from  \cite[Equations~(3),(5)]{vershik1985asymptotic} becomes $0$ when replacing $\mu _N$ with $\mu _N'$. Using the notations $\theta _{\lambda }$ and $\tilde \theta (\lambda )$ of \cite[Equations~(5)]{vershik1985asymptotic} becomes
    \begin{equation}\label{eq:mu}
        -\log (\mu _N'(\lambda ))= N \theta _\lambda  +\sqrt{N}\tilde \theta (\lambda ).
    \end{equation}
    The formula \eqref{eq:VK} then follows by substituting  \cite[Lemmas~1,2,4]{vershik1985asymptotic} into \eqref{eq:mu}. Indeed, \cite[Lemma~1]{vershik1985asymptotic} gives the formula for $\tilde \theta (\lambda )$. Then, \cite[Lemma~4]{vershik1985asymptotic} with $f=L_{\lambda }-\Omega $ substituted into \cite[Lemma~2]{vershik1985asymptotic} and then into \eqref{eq:mu} gives the formula for $\theta _{\lambda }$. Note that $\tilde \theta (\lambda )$ of \cite{vershik1985asymptotic} is our $\tilde \theta (f)$ divided by $\sqrt{N}$. Also note that \cite{vershik1985asymptotic} work with step functions $L_\lambda $ whose steps are of size $1/\sqrt{4N}$, while our functions $f$ have steps of size $1$. Hence, in order to obtain \eqref{eq:VK}, we rescale by $f(x)=\sqrt{4N}L_{\lambda }(x/\sqrt{4N})$. Writing $L_{\lambda }-\Omega $ in \cite[Lemma~4]{vershik1985asymptotic} instead of $f$ and then changing variables in the second integral by $s=x/\sqrt{4N}$ leads to our $\bar \theta _N(f)$ and in the first integral by $s=x/\sqrt{4N}$ and $t=y/\sqrt{4N}$ leads to our $\theta _N(f)$.
\end{proof}

\subsection{Proof outline}

\label{Proof0} 

By Proposition~\ref{prop:VK}, to find a partition of maximal dimension, it suffices to find a function $f\in \mathcal S_N$ minimizing $\frac{1}{8}\theta _N(f) +\tilde \theta (f) +\bar \theta _N(f)$. We will see at this minimizer that $\bar \theta _N(f)$ is negligible so, to prove Theorem~\ref{thm:main}, one must show that $\frac{1}{8}\theta _N(f) +\tilde \theta (f)=(\mathfrak{d}+o(1))\sqrt{N}$, for some constant $\mathfrak{d}$. We will also see at this minimizer that $\theta _N(f)$ and $\tilde \theta (f)$ are of order $\sqrt{N}$, and that these two terms are effectively in competition with each other. 

Indeed, requiring that $\theta_N (f)$ be small forces $f$ to approximate $\Omega_N$. However, since we cannot have $f = \Omega_N$ (as $\Omega_N$ is curved, while $f$ is associated with a Young diagram), the minimum of $\theta_N (f)$ alone is intuitively attained when $f$ is some rounding of $\Omega_N$. Already in this case, one can verify that $\theta_N (f) \ge c\sqrt{N}$, where the constant $c>0$ depends intricately on this precise choice of rounding of $\Omega_N$ to $f$. Minimizing the sum $\frac{1}{8} \theta_N (f) + \tilde{\theta} (f)$ is more subtle, as making $\tilde{\theta}$ small incentivizes the partition $\lambda(f)$ to have few small hook lengths, causing $f$ to locally look quite ``boxlike'' and thus different from a genuine rounding of $\Omega_N$.

To analyze this optimization problem, we zoom in locally and estimate the optimal contribution to $\frac{1}{8}\theta _N(f)+\tilde \theta (f)$ coming from an interval of length $n\ll \sqrt{N}$ around some point $z$; see Figure~\ref{fig:2}. Near $z$, the ratio $\frac{\Omega _N(x)-\Omega _N(y)}{x-y}$ appearing in $\theta _N(f)$ can be replaced with a fixed slope $\rho =\Omega _N'(z)$, to create the corresponding local functional $\theta _n^{\rho }$; see \eqref{eq:localthetas}. This leads to the local optimization problem considered in Section~\ref{sec:local}. The optimal local contribution to $\frac{1}{8}\theta _N(f) +\tilde \theta (f)$ coming from functions of length $n$ at slope $\rho $ is denoted by $
\sigma _n^{\rho }$ and is shown in Section~\ref{sec:local} to approximate $\sigma ^{\rho }n$, for some constant $\sigma ^{\rho }$. Thus, this optimal local contribution will be linear in the local pattern's length, with a specific constant depending on its slope. 

In Section~\ref{sec:global}, we show that the global sum $\frac{1}{8} \theta_N (f) + \tilde{\theta} (f)$ can be approximated by an integral of local contributions. Heuristically, this should arise from concatenating the above local optimizers around a fine mesh of points $(z_i) \subset [-\sqrt{4N},\sqrt{4N}]$. Indeed, if these optimizers did not interact, then the global sum could be approximated as $\frac{1}{8} \theta_N (f) + \tilde{\theta} (f) \approx \sum_i \sigma_n^{\Omega_N' (z_i)} \approx n \sum_i \sigma^{\Omega_N' (z_i)} \approx \sqrt{N} \int_{-1}^1 \sigma^{\Omega'(x)} dx$, leading to the global constant $\mathfrak d=\int _{-1}^{1}\sigma ^{\Omega '(x)}dx$. 

However, these concatenated optimizers do exhibit long-range interactions, arising as terms in the double integral $\theta _n^{\rho }(f)$ from distant pairs of points $(x,y)$ near different $z_i$. To bound them directly, one must show that the optimal local function is very flat, namely, that it is very close (of distance $n^{1/2-c}$) to the line $y=\rho x$. For specific discrete log gases, this flatness (also called rigidity) can be shown using the loop equations, as studied in works by Borodin, Gorin, Guionnet, and Huang \cite{AODE,EDE}. Since the loop equations are not known (or predicted) to hold in our setup, we did not see a direct proof of this strong form of flatness. 

To circumvent this, we instead show that there exists a dyadic scale $m \ll n$ and a nearly optimal local function on an interval of length $m$, which satisfies certain averaged flatness properties (listed in Lemma \ref{lem:12}). The latter enable us to concatenate this ``seed'' near-optimizer $n/m$ times (Proposition~\ref{prop:1}) to form a near-optimizer on the original interval of length $n$; see Figure \ref{fig:4}. This near-optimizer will then be very flat and can therefore be concatenated (with negligible interaction) along $\Omega_N$, as outlined above. The idea of concatenating local near-optimizers on smaller scales to form very flat ones on larger scales was also used by Sandier--Serfaty \cite{LRE} in the context of continuum log gases. They produced their ``seed'' local near-minimizers by embedding the original (nonlocal) one-dimensional log gas into a two-dimensional Coulomb gas \cite{GAR}, whose interactions were more local. 

It seems less transparent that an embedding with such a locality property should exist in our discrete setting. So, we instead specify this ``seed'' through a different, more combinatorial route, by restricting the true local minimizer $f$ to a dyadic scale on which it is flat. The reason such a scale exists is that the double integral $\theta _n^{\rho }(f)$ will be linear in $n$. Thus, since there are $c \log n$ dyadic scales up to size $n$, there will be a scale of distances of $|x-y|$ that will contribute at most $Cn (\log n)^{-1}$ to the double integral (which will be the source of the error term $C(\log N)^{-1}$ in Theorem~\ref{thm:main}). At that scale, the function $f$ must be somewhat flat on average; we can use the restriction of $f$ to an appropriately chosen interval at this scale as our seed that will be concatenated to form very flat near-minimizers.

\subsection{Policy on constants}
We use $C,c$ to denote positive universal constants. These constants are regarded as generic in the sense that their value may change from one appearance to the next, with the value of $C$ increasing and the value of $c$ decreasing.

We write $A=O(B)$ if $|A|\le C|B|$ for some universal constant $C$.

\section{Local optimizers}\label{sec:local}

We start by setting up the local optimization problem we consider. For $n\ge 1$ let $\mathcal S_n^{\rm loc}$ be the set of continuous, piecewise linear functions $f:[0,n]\to \mathbb R$ with $f(0)=0$ and $f'(x)\in \{-1,1\}$ for any $x\notin \mathbb Z$.

\begin{figure}[ht]
\includegraphics[width=1\linewidth]{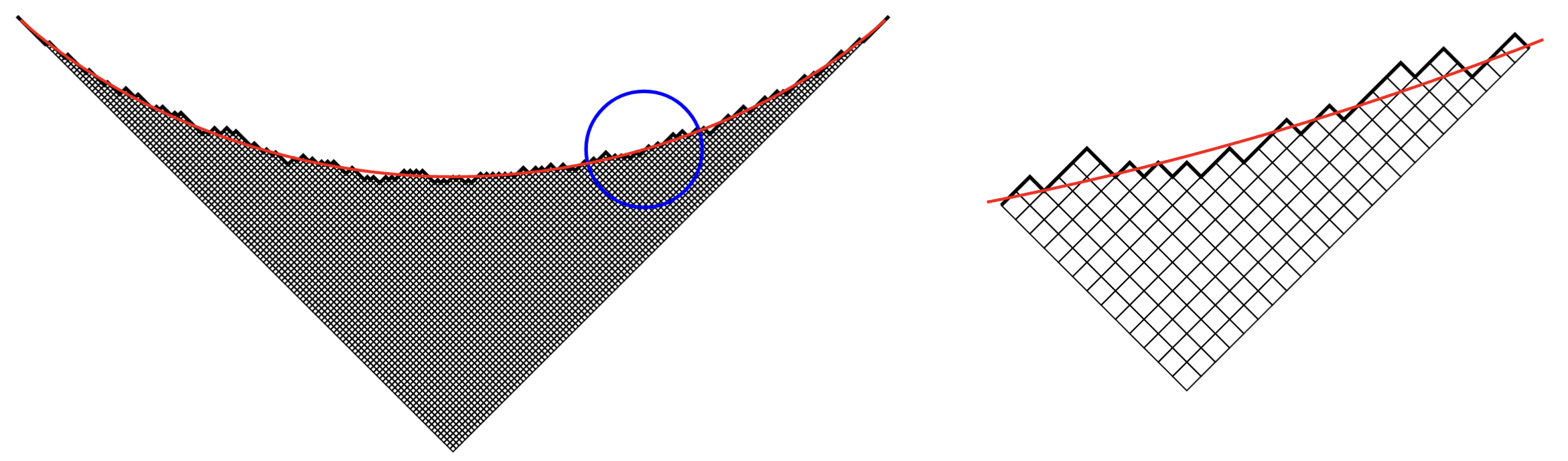}
\caption{An approximate minimizer of the global optimization problem can be constructed by concatenating approximate local minimizers along the curve $\Omega_N$. When zooming in locally around $z$, the function $\Omega_N$ can be replaced with a straight line with slope $\rho=\Omega_N'(z)$. This leads to the local optimization problem in Section~\ref{sec:local}.}
\label{fig:2}
\end{figure}

Let $\rho \in [-1,1]$ be a slope. For $n\ge 1$ and $f\in \mathcal S_n^{\rm loc}$ define 
\begin{equation}\label{eq:localthetas}
    \theta _n^{\rho }(f):= \int _0^n \int _0^n \Big( \frac{f(x)-f(y)}{x-y}-\rho  \Big)^2 dx\, dy, \quad \tilde \theta _n(f):=\sum _{(i,j)} \varphi (h_{i,j})
\end{equation}
where the sum is over cells $(i,j)$ for which the hook length $h_{i,j}$ is well defined. Namely, these are the cells (rotated squares of side length $\sqrt{2}$) that are trapped below $f(x)$ and above the function $\max (-x,x-n+f(n))$. These are precisely the cells that appear in the right pane in Figure~\ref{fig:2} (with $f$ being the upper boundary of the diagram and $(0,0)$ is the left most point of the diagram). The function $\varphi $ is defined as in \eqref{eq:tilde1}. Let $$\Theta _n^{\rho } (f):=\frac{1}{8}\theta _n^{\rho } (f)+\tilde \theta _n (f)$$ and let 
\begin{equation}
    \sigma _n^{\rho }:=\min _{f\in \mathcal S_n^{\rm loc}} \Theta_n^{\rho } (f).
\end{equation}

The main result of this section is given in the following theorem.

\begin{thm}\label{thm:local}
    For any $\rho \in [-1,1]$, there is a constant $\sigma ^{\rho }\ge 0$ such that 
    \begin{equation}\label{eq:99}
        \Big| \frac{\sigma _n^{\rho }}{n}-\sigma^{\rho } \Big| \le \frac{C}{\log n}.
    \end{equation}
    Moreover, the function $\rho \mapsto\sigma ^{\rho }$ is continuous on $[-1,1]$.
\end{thm}

To prove a quantitative bound on the convergence $\frac{\sigma _n^{\rho }}{n}\to \sigma ^{\rho }$ we prove that $\sigma _n^{\rho }$ is superadditive and approximately subadditive.

\begin{lem}\label{lem:super}
    For any $\rho \in [-1,1]$ and $n,m\ge 1$ we have that $\sigma _{n+m}^{\rho}  \ge \sigma _n^{\rho }+\sigma_m^{\rho }$.
\end{lem}

\begin{proof}
    Let $f\in \mathcal S^{\rm loc}_{n+m}$ be the function achieving the minimum in the definition of $\sigma _{n+m}^{\rho }$. Let $g\in \mathcal S^{\rm loc}_n$ defined by $g(x)=f(x)$ for any $x\in [0,n]$ and $h\in \mathcal S^{\rm loc}_m$ defined by $h(x)=f(n+x)-f(n)$. We clearly have that $\theta _{n+m}^{\rho }(f)\ge \theta _n^{\rho }(g)+\theta _m^{\rho }(h) $ and $\tilde \theta _{n+m}(f)\ge \tilde \theta _{n}(g)+\tilde \theta _m(h)$ since we just integrate and sum over bigger domains.
\end{proof}

\begin{lem}\label{lem:11}
    We have that $cn(1-|\rho|)^2 \le \sigma _n^{\rho }\le Cn(1-|\rho |)\log \big( \frac{2}{1-|\rho|} \big)$.
\end{lem}

\begin{proof}
    For the lower bound observe that when $x,y\in [i,i+1]$ for some $i\le n$ we have $\frac{f(x)-f(y)}{x-y}=\pm 1 $ and therefore $\big( \frac{f(y)-f(x)}{y-x}-\rho  \big)^2\ge (1-|\rho|)^2$. The contribution to the integral from such pairs $(x,y)$ will be of order $(1-|\rho|)^2n$.
    
    We turn to prove the upper bound. By symmetry, we may assume that $\rho \ge 0$. We let $f$ be the maximal function in $\mathcal S^{\rm loc}_n$ for which $f(x)\le \rho x$ for all $x\in [0,n]$ (the maximum of functions in $\mathcal S^{\rm loc}_n$ is in $\mathcal S^{\rm loc} _n$ so this is well defined). Note that $f(x)\in [\rho x-2,\rho x]$ for all $x\in [0,n]$. This is a function that takes one down step every $\approx 2/ (1-\rho )$ steps. Let $A:=\{(x,y)\in [0,n]^2: f(y)=f(x)+(y-x)\}$ be the set of pairs $x,y$ that are in the same continuous interval of up steps. There are at most $C(1-\rho )n$ such intervals, each of length at most $C/(1-\rho )$. Moreover, for $(x,y)\in A$ we have that $(\frac{f(x)-f(y)}{x-y}-\rho )^2=(1-\rho )^2$ and therefore the contribution to $\theta _n^{\rho }$ coming from $(x,y)\in A$ is bounded by $C(1-\rho )n (\frac{1}{1-\rho })^2 (1-\rho )^2=C(1-\rho )n$. Next, we have that $|f(x)-\rho x|\le 2$ and therefore for $(x,y)\notin A$ we have    $(\frac{f(x)-f(y)}{x-y}-\rho )^2\le \frac{C}{(|x-y|+1)^2}$. Thus, the contribution from such pairs is bounded by 
    \begin{equation}
        \int _{[0,n]^2\setminus A}\frac{C}{(|x-y|+1)^2} dx\, dy \le \int _0^n \frac{C}{d_x+1} dx \le Cn(1-\rho )\log \Big( \frac{2}{1-\rho} \Big),
    \end{equation}
    where $d_x:=\inf \{|x-y|: f'(y)=-1\}$ is the distance of $x$ from the closest down step. We obtain that $\theta _n^{\rho }(f)\le  Cn(1-\rho )\log \big( \frac{2}{1-\rho} \big)$.

    Next, we turn to bound $\tilde \theta  _n(f)$. First, note that for any integer $x$
    \begin{equation}
        \varphi (x)\le x^{-2} \sum _{k=1}^{\infty } \frac{1}{k(k+1)(2k+1)}  \le Cx^{-2}.
    \end{equation}
    Moreover, for any integer $x\le \frac {1}{1-\rho }$ the number of cells $(i,j)$ for which $h_{i,j}=x$ is bounded by $Cn(1-\rho )$. When $x\ge 1/(1-\rho )$ we can simply bound this number by $Cn$. We obtain that 
    \begin{equation}
        \tilde \theta _n(f) \le \sum _{x\le 1/(1-\rho )} Cn(1-\rho )x^{-2} + \sum _{x\ge 1/(1-\rho )}Cnx^{-2} \le C(1-\rho )n.
    \end{equation}
    This finishes the proof of the upper bound.
\end{proof}

From the last two lemmas and Fekete's lemma it follows that $\sigma _n^{\rho }/n$ converges and we let $\sigma ^{\rho }:=\lim _{n\to \infty } \sigma _n^{\rho }/n$. It also follows that
\begin{equation}\label{eq:23}
    \frac{\sigma _n^{\rho }}{n}\le \sigma ^{\rho }
\end{equation}
since by Lemma~\ref{lem:super} we have that $\sigma _n^{\rho }/n\le \sigma _{kn}^{\rho }/{kn} \to \sigma ^{\rho }$ as $k\to \infty $. Hence, in order to prove Theorem~\ref{thm:local} it suffices to prove that $\sigma _n^{\rho }/n \ge \sigma ^{\rho }-C/\log n$. The proof of this bound is concluded in Proposition~\ref{prop:1} below.

\subsection{Fluctuations}

\begin{lem}\label{lem:flat}
    Let $\rho \in [-1,1]$ and let $f\in \mathcal S^{\rm loc}_n$ such that $\theta _n^{\rho } (f)\le Bn$ for some $B> 0$. Then, there is $C_B>0$ depending only on $B$ such that $|f(x)-\rho x|\le C_Bn^{2/3} $ for any $x\in [0,n]$.
\end{lem}

Let us note that using a simple inductive argument (or the fractional Sobolev inequality) one can replace the exponent $2/3$ in the lemma by $1/2+\epsilon$, for any $\epsilon>0$. However, this will not be needed for our proof.

\begin{proof}
    Suppose otherwise and without loss of generality, assume that  $f(x_0)\ge \rho x_0 + Mn^{2/3} $ for some $x_0\in [0,n]$ and a very large constant $M$ (to be determined depending on $B$ later). Let $I:=\{x\in [0,n]: f(x)\le \rho x +\frac{1}{2}Mn^{2/3} \}$. We consider separately two cases:

    Case 1: We have that $|I|\ge n/2$. In this case we let $J:=\{y\in [0,n]: |y-x_0|\le \frac{M}{8}n^{2/3}\}$. Since $|f'(y)|\le 1$ for all $y$, for any $y\in J$ 
    \begin{equation}
        f(y)\ge f(x_0)-\frac{M}{8}n^{2/3} \ge \rho x_0 +\frac{7}{8} Mn^{2/3}\ge   \rho y +\frac{3}{4}Mn^{2/3}
    \end{equation}
    Hence, for any $x\in I$ and $y\in J$ we have that 
    \begin{equation}
    \Big( \frac{f(y)-f(x)}{y-x}-\rho \Big)^2 \ge \Big( \frac{cMn^{2/3}}{y-x} \Big) ^2 \ge cM^2n^{-2/3}.
    \end{equation}
    Thus, the contribution to the integral in $\theta _n^{\rho }(f)$ from $x\in I$ and $y\in J$ will be at least $cM^2n^{-2/3}|I||J|\ge cMn$ which contradicts our assumption on $\theta _n^{\rho }(f)$ when $M$ is sufficiently large. 

    Case 2: We have that $|I|\le n/2$. In this case, we consider the contribution to the integral of $\theta _n^{\rho }(f)$ coming from $x\in I^c$ and $y\in [0,\frac{M}{8}n^{2/3}]$. For such $y$ we have that $f(y)\le \rho y +\frac{M}{4}n^{2/3}$ so as in Case 1 this contribution will be at least $cMn$.
\end{proof}

\subsection{Construction of a flat almost minimizer}

In this subsection and the next we use an inductive argument in order to finish the proof of Theorem~\ref{thm:local}. To this end, let $A$ be a sufficiently large constant to be determined later. The constants $C,c$ are not allowed to depend on $A$, and $A$ will be chosen sufficiently large depending on these constants. Observe that when $A$ is  sufficiently large, for all $m\le e^{\sqrt{A}}$ and $\rho \in [-1,1]$ we have that 
\begin{equation}
    \frac{\sigma _m^{\rho }}{m} \ge \sigma^{\rho } -\frac{A}{\log m}
\end{equation}
since the right hand side will be negative in this case.

Next, we let $n\ge e^{\sqrt{A}}$ and assume by induction that for any  $\rho\in [-1,1]$ we have 
\begin{equation}\label{eq:induction}
    \frac{\sigma _m^{\rho }}{m} \ge \sigma^{\rho } -\frac{A}{\log m},\quad \text{for all $m<n$}.
\end{equation}

To complete the induction we need to prove that the same inequality holds when $m=n$, which is done in Proposition~\ref{prop:1} below. First, observe that when $1-|\rho|\le n^{-0.01}$ we have by Lemma~\ref{lem:11} that $\sigma ^{\rho }\le Cn^{-0.01}\log n$ and therefore \eqref{eq:induction} with $m=n$ holds trivially in this case (the right hand side will be negative). Hence, from now on we fix a slope $\rho$ with $1-|\rho |\ge n^{-0.01}$.

In the next key lemma, we construct an almost minimizer $g$ that is somewhat flat. Then, in the proof of Proposition~\ref{prop:1} below, we will concatenate this almost minimizer to itself to create a function on a larger domain having a small $\Theta ^{\rho }$.

\begin{lem}\label{lem:12}
    There exists $m\in [n^{1/3},n^{2/3}]$ such that for any $a\in \mathbb Z$ with $a\equiv m (\!\!\!\mod 2)$ and $|a-\rho m|\le 10$, there is a function $g\in \mathcal S^{\rm loc}_m$ with $g(m)=a$ such that
    \begin{enumerate}
        \item We have 
        \begin{equation}\label{eq:theta62}
            \frac{\Theta _m^{\rho }(g)}{m} \le \frac{\sigma _n ^{\rho }}{n} + \frac{CA^{2/3}}{\log n} .
        \end{equation}
        \item Letting $\bar g(x)=g(x)-\rho x$ we have
        \begin{equation}\label{eq:integral}
            \int _0^{m}  \frac{\bar g(x)^2}{x} dx \le \frac{CA^{2/3}m}{\log n} \quad \text{and} \quad  \int _0^{m}  \frac{(\bar g (m)-\bar g(x))^2}{m-x} dx \le \frac{CA^{2/3}m}{\log n}  
        \end{equation}
    \end{enumerate}
\end{lem}

\begin{proof}
Let $f\in \mathcal S^{\rm loc}_n$ be the function achieving the minimum in the definition of $\sigma _n^{\rho }$. Our function $g$ will be constructed by linearly tilting a restriction of $f$ to a carefully chosen interval $[i,i+m]\subseteq [0,n]$. The proof contains four steps: choosing the dyadic scale of $m$, choosing the exact length $m$, choosing the right starting point $i$ and adding a linear shift to the restriction.

{\bf Step 1: Choosing the scale.}
We choose our dyadic scale $k$ in the interval $[k_0,k_1]$ where $k_0:=\lceil \log _2(n^{1/3})\rceil +1$ and $k_1:=\lfloor \log _2 (n^{2/3}) \rfloor -1$. For $k\le k_1$ define  
    \begin{equation}
        I_k(f):= \int _{\substack{x, y\in [0,n] \\ 2^{k} < |x-y|\le 2^{k+1}}} \Big( \frac{f(x)-f(y)}{x-y}-\rho  \Big)^2 dx \, dy .
    \end{equation}
   We have that 
\begin{equation}\label{eq:scale}
    \begin{split}
        \sum _{k=1}^{k_1} 2^{-k/3} \sum _{\ell=1}^{k} 2^{\ell/3} I_\ell(f)&= \sum _{\ell=1}^{k_1} 2^{\ell/3} I_\ell(f) \sum _{k=\ell }^{k_1}2^{-k/3}\\
        &\le \sum _{\ell =1}^{k_1} 2^{\ell/3} I_\ell (f) C2^{-\ell/3} =C\sum _{\ell =1}^{k_1} I_\ell (f)\le C\theta _n^{\rho }(f) \le  Cn,
    \end{split}
    \end{equation}
    where the last inequality follows from Lemma~\ref{lem:11}. It follows that there exists $k\in [k_0,k_1]$ such that
    \begin{equation}\label{eq:scale2}
        \sum _{\ell =1}^{k}2^{\ell/3} I_\ell  (f) \le C\frac{2^{k/3}n}{\log n},
    \end{equation}
    since otherwise the sum on the left hand side of \eqref{eq:scale} would be too large contradicting \eqref{eq:scale}. We fix such a scale $k\in [k_0,k_1]$ and note that in particular we have $I_k(f)\le Cn/(\log n)$.
    
    {\bf Step 2: Choosing the exact length.}
    Letting $\bar f(x):=f(x)-\rho x$ we have that 
    \begin{equation}
    \begin{split}
        \sum _{m=2^k}^{2^{k+1}} \sum _{i=0}^{n-m} \big( \bar f(i+m)-\bar f(i) \big) ^2 &\le C2^{k}n + 2\int _{2^k}^{2^{k+1}} \int _{0}^{n-r} \big( \bar f(x+r)-\bar f(x) \big) ^2 dx\, dr \\
        &\le C2^{k}n +C4^{k} I_k(f) \le \frac{C4^kn}{\log n},
    \end{split}
    \end{equation}
    where in the first inequality we used that $(\bar f(i+m)-\bar f(i))^2 \le 2(\bar f(x+r)-\bar f(x))^2$ for any $x\in [i,i+1]$ and $r\in [m,m+1]$ unless $|\bar f(i+m)-\bar f(i)|\le 7$ and such $i,m$ will contribute at most $C2^kn$ to the sum. It follows that there exists $m\in [2^k,2^{k+1}]$ such that 
\begin{equation}\label{eq:1}
    \sum _{i=0}^{n-m} (\bar f(i+m)-\bar f(i))^2 \le \frac{C2^kn}{\log n} \le \frac{Cnm}{\log n}.
\end{equation}
    This will be the length $m$ from the statement of the lemma.

    {\bf Step 3: Choosing the starting point.}
    Next, for $i\le n-m$ let $g_i\in \mathcal S_m^{\rm loc}$ defined by $g_i(x):=f(i+x)-f(i)$ for $x\in [0,m]$. We would like to choose an appropriate $i\le n-m$ and define the function $g$ to be a linear shift of $g_i$. To this end, we define a few sets of bad integers $i$, and then choose an $i$ outside these sets. The sets $B_1,B_2,B_3,B_4$ will measure various ways in which $\bar f$ might not be flat on the interval $[i,i+m]$; the complements of $B_2,B_3$ will ensure that \eqref{eq:integral} is satisfied and the complement of $B_5$ will ensure that \eqref{eq:theta62} is satisfied. Define 
    \begin{equation}
    \begin{split}
        B_1&:= \bigg\{ i\le n-m: |\bar f(i+m)-\bar f(i)|\ge  \sqrt{A^{2/3}\frac{m}{\log n}} \bigg\}\\
        B_2&:=\bigg\{ i\le n-m: \int _i^{i+m}\frac{(\bar f(x)-\bar f(i))^2}{x-i} dx \ge A^{2/3} \frac{m}{\log n} \bigg\}\\
        B_3&:=\bigg\{ i\le n-m: \int _i^{i+m}\frac{(\bar f(i+m)-\bar f(x))^2}{i+m-x} dx \ge A^{2/3} \frac{m}{\log n} \bigg\}\\
        B_4&:=\bigg\{ i\le n-m: \int _i^{i+m}\int _i^{i+m}\frac{(\bar f(x)-\bar f(y))^2}{|x-y|^{5/3}} dx \, dy \ge A^{2/3} \frac{m^{4/3}}{\log n} \bigg\}\\
        B_5&:=\bigg\{ i\le n-m : \frac{\Theta _m^{\rho }(g_i)}{m} \ge  \frac{\Theta _n^{\rho }(f)}{n}+\frac{A^{2/3}}{\log m}  \bigg\}.
    \end{split}
    \end{equation}

    Next we bound the sizes of $B_1,B_2,B_3,B_4,B_5$. It follows from \eqref{eq:1} that $|B_1|\le CA^{-2/3}n$.

    We turn to bound the size of $B_2$. We have that 
\begin{equation}
\begin{split}
    \sum _{i=0}^{n-m} \int _i^{i+m}\frac{(\bar f(x)-\bar f(i))^2}{x-i}dx &\le Cn\log m+ 2\int _0^{n-m} \int _{y+2}^{y+m} \frac{\big( \bar f(x)-\bar f(y) \big) ^2}{x-y} dx \, dy \\
    &\le Cn\log m + C\sum _{\ell =1}^{k} \int _{\substack{0\le y\le x\le n \\ 2^{\ell } < x-y\le 2^{\ell +1}}}  \frac{\big( \bar f(x)-\bar f(y) \big) ^2}{x-y} dx \, dy \\
    &\le Cn\log m + C\sum _{\ell =1}^{k} 2^\ell I_\ell (f) \le \frac{Cnm}{\log n},
\end{split}
\end{equation}
where in the first inequality we used that $(\bar f(x)-\bar f(i))^2\le 2 (\bar f(x)-\bar f(y))^2$ for any $y\in [i,i+1]$ unless $|\bar f(x)-\bar f(i)|\le 3$ and the contribution from such $(i,x)$ to the sum will be at most $Cn\log m$. The last inequality follows from \eqref{eq:scale2} and the fact that $m\ge 2^k$. From this it follows that $|B_2|\le CA^{-2/3}n$. By the same arguments we also have that $|B_3|\le CA^{-2/3}n$.

Similarly, to bound $B_4$ we write
\begin{equation}
\begin{split}
    \sum _{i=0}^{n-m} \int _i^{i+m}\int _i^{i+m}\frac{(\bar f(x)-\bar f(y))^2}{|x-y|^{5/3}} dx \, dy \le 2m \int _0^{n-m}\int _y^{y+m} \frac{(\bar f(x)-\bar f(y))^2}{(x-y)^{5/3}} dx \, dy \\
    \le Cmn + Cm\sum _{\ell =1}^{k} 2^{\ell /3} I_\ell (f) \le Cmn +Cm\frac{2^{k/3}n}{\log n} \le \frac{Cm^{4/3}n}{\log n}.
\end{split}
\end{equation}
where in the first inequality we used that each pair $(x,y)$ is counted at most $m$ times in the sum, in the second inequality the term $Cmn$ comes from pairs $x>y$ with $x-y\le 2$, and in the third inequality we used \eqref{eq:scale2}. It follows that $|B_4|\le CA^{-2/3}n$.

Finally, we bound the size of $B_5$. We have that 
\begin{equation}
    \sum _{i=1}^{n-m} \Theta _m^{\rho }(g_i) \le m\Theta _n^{\rho }(f),
\end{equation}
since any cell $(i,j)$ in the definition of $\tilde \theta _n$ and similarly, any pair $(x,y)$ in the integral in $\theta _n^{\rho }$, will be counted at most $m$ times in the sum $ \sum _{i=1}^{n-m} \Theta _m^{\rho }(g_i)$. Moreover, by the inductive hypothesis and \eqref{eq:23}, for all $i\le n-m$ we have
$$
\frac{\Theta _m^{\rho }(g_i)}{m}\ge \sigma ^{\rho } -\frac{A}{\log m}\ge \frac{\sigma _n^{\rho }}{n}-\frac{A}{\log m}= \frac{\Theta _n^{\rho }(f)}{n}-\frac{A}{\log m}.
$$ Thus, if $|B_5|\ge n/2$ then
\begin{equation}
\begin{split}
     \Theta _n^{\rho }(f) \ge \sum _{i=1}^{n-m} \frac{\Theta _m^{\rho }(g_i)}{m} &\ge |B_5| \Big( \frac{\Theta _n^{\rho }(f)}{n}+\frac{A^{2/3}}{\log m} \Big)+ (n-m-|B_5|) \Big( \frac{\Theta _n^{\rho }(f)}{n} -\frac{A}{\log m} \Big)\\
     &= (n-m) \frac{\Theta _n^{\rho }(f)}{n} -  \frac{A(n-m)}{\log m} +|B_5| \Big( \frac{A}{\log m}+ \frac{A^{2/3}}{\log m} \Big) \\
     &\ge \Theta _n^{\rho }(f) -\frac{An}{\log m} +|B_5| \Big( \frac{A}{\log m} +\frac{A^{2/3}}{2\log m} \Big),
\end{split}
\end{equation}
where in the last inequality the term $m\frac{\Theta _n^{\rho }(f)}{n}$ can be absorbed in the error term since $m\le n^{2/3}$ and using Lemma~\ref{lem:11}. Hence, in any case we obtain that $|B_5|\le \frac{A}{A+\frac{1}{2}A^{2/3}}n \le (1-\frac{1}{4}A^{-1/3})n$, where the last inequality holds when $A$ is sufficiently large.

Combining the bounds on $B_1,B_2,B_3,B_4,B_5$ we obtain that as long as $A$ is sufficiently large, there is $i\le n-m$ such that $i\notin B_1\cup B_2\cup B_3\cup B_4\cup B_5$. The function $g_i\in \mathcal S^{\rm loc}_m$ satisfies \eqref{eq:theta62} and \eqref{eq:integral} since $i\notin B_2\cup B_3\cup B_5$. However, we might not have  $g_i(m)=a$.

{\bf Step 4: A linear shift.}
In the last step we add an (approximately) linear shift to $g_i$ to form $g$ with $g(m)=a$. We will then show that $g$ retains the bounds \eqref{eq:theta62} and \eqref{eq:integral} using that $i\notin B_1\cup B_4$. Without loss of generality, suppose that $g_i(m)>a$ and let $s:=(g_i(m)-a)/2\le A^{1/3}\sqrt{m/\log n}$, where the last inequality follows since $i\notin B_1$. We claim that there are integers $m/3\le j_1<\dots <j_s\le 2m/3$ such that $j_ \ell \ge j_{\ell -1}+m/(10s)$ and such that $g_i(j_\ell )<g_i(j_\ell+1)$. To see this, we let $J_1,\dots , J_{2s}$ be disjoint intervals (with integer endpoints) of length $\lceil m/(10s)\rceil $, separated by distance $m/(10s)$ from each other. Since $i\notin B_5$ and by Lemma~\ref{lem:11} we have $\theta_m ^{\rho }(g_i)<Cm$ and therefore for at least half of these intervals $J_\ell$, we have 
\begin{equation}\label{eq:3.24}
    \int _{J_\ell }\int _{J_\ell } \bigg( \frac{\bar g_i(x)-\bar g_i(y)}{x-y} \bigg)^2dx\,dy \le C\frac{m}{s}.
\end{equation}
Call these intervals $J_1',\dots ,J_s'$. Since 
$1-|\rho |\ge n^{-0.01}\ge m^{-0.03}$ and $n\ge e^{\sqrt{A}}$ we have that $20C_B(m/s)^{2/3}<(1-|\rho |)(m/s)$ where $C_B$ is from Lemma~\ref{lem:flat}, with $B$ determined by \eqref{eq:3.24}. Thus, by Lemma~\ref{lem:flat}, applied to $g_i$ restricted to the interval $J_\ell '$, it follows that the function remains at distance $(1-|\rho |)|J_\ell '|/2$ from a straight line with slope $\rho $ and so it cannot take only down steps in this interval. Hence, there must be $j_\ell \in J_\ell '$ such that $g_i(j_\ell)<g_i(j_\ell +1)$.

\begin{figure}[ht]
\centering
\begin{tikzpicture}[scale=0.4]

    \draw[->] (-0.5,0) -- (25.8,0);

    \draw (5,0.18) -- (5,-0.18);
    \draw (13,0.18) -- (13,-0.18);
    \draw (20,0.18) -- (20,-0.18);

    \node[below] at (5,-0.18) {$j_1$};
    \node[below] at (13,-0.18) {$j_2$};
    \node[below] at (20,-0.18) {$j_3$};

    \draw[thick]
        (0,0) --
        (1,1) --
        (2,2) --
        (3,3) --
        (4,2) --
        (5,3) --
        (6,4) --
        (7,3) --
        (8,4) --
        (9,3) --
        (10,4) --
        (11,5) --
        (12,4) --
        (13,5) --
        (14,6) --
        (15,7) --
        (16,6) --
        (17,7) --
        (18,8) --
        (19,7) --
        (20,8) --
        (21,9) --
        (22,8) --
        (23,9) --
        (24,10) --
        (25,9);

    \foreach \x/\y in {
        0/0,1/1,2/2,3/3,4/2,5/3,6/4,7/3,8/4,9/3,
        10/4,11/5,12/4,13/5,14/6,15/7,16/6,17/7,
        18/8,19/7,20/8,21/9,22/8,23/9,24/10,25/9
    }
        \fill (\x,\y) circle (2pt);

    \draw[thick,dashed]
        (0,0) --
        (1,1) --
        (2,2) --
        (3,3) --
        (4,2) --
        (5,3) --
        (6,2) --
        (7,1) --
        (8,2) --
        (9,1) --
        (10,2) --
        (11,3) --
        (12,2) --
        (13,3) --
        (14,2) --
        (15,3) --
        (16,2) --
        (17,3) --
        (18,4) --
        (19,3) --
        (20,4) --
        (21,3) --
        (22,2) --
        (23,3) --
        (24,4) --
        (25,3);

    \foreach \x/\y in {
        0/0,1/1,2/2,3/3,4/2,5/3,6/2,7/1,8/2,9/1,
        10/2,11/3,12/2,13/3,14/2,15/3,16/2,17/3,
        18/4,19/3,20/4,21/3,22/2,23/3,24/4,25/3
    }
        \fill (\x,\y) circle (2pt);

    \node[right] at (14.3,7.7) {$g_i$};
    \node[right] at (14.3,3.7) {$g$};

    \draw (26.2,9) -- (26.6,9);
    \draw (26.2,3) -- (26.6,3);
    \node[right] at (26.6,9) {$g_i(m)$};
    \node[right] at (26.6,3) {$a$};

\end{tikzpicture}
\caption{The function $g$ takes the same up down steps as $g_i$ except for the steps at $j_1,\dots ,j_s$ in which $g$ goes down instead of up. This modification will make sure that $g(m)=a$ without increasing $\Theta _m^{\rho }$ by much.}
\label{fig:3}
\end{figure}
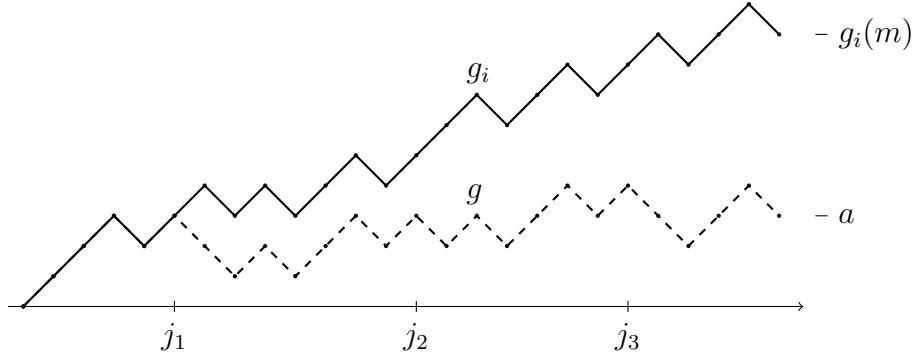

Once we found these integers $j_1,\dots ,j_s$, we let $g\in \mathcal S^{\rm loc}_m$ be the function that takes exactly the same up down steps as $g_i$, but takes a down step at $j_\ell$ instead of an up step (See Figure~\ref{fig:3}). We have that $g(m)=a$ and therefore it suffices to verify that this $g$ still satisfies the other requirements. To this end, we let $h:=g-g_i$ and write 
\begin{equation*}
    \theta _m^{\rho }(g)=\theta _m^{\rho }(g_i) + 2I_1 +I_2,
\end{equation*}
where
\begin{equation*}
    I_1:=\int_0^m \int_0^m  \frac{(h(x)-h(y))(\bar g_i(x)-\bar g_i(y))}{(x-y)^2}dx\, dy ,\quad I_2:=\int _0^m\int _0^m \Big(\frac{h(x)-h(y)}{x-y}\Big)^2 dx\, dy.
\end{equation*}

We start by bounding $I_2$. We have that $|h(x)-h(y)|\le Cs|x-y|/m$ when $|x-y|\ge m/s$ and $|h(x)-h(y)|\le C$ when $|x-y|\le m/s$. Moreover $h(x)=h(y)$ as long as there is no $j_\ell$ such that $x<j_\ell <y$. Thus, we have that 
\begin{equation*}
I_2 \le s\int _0^{m/s}\int _0^{m/s} \frac{C}{(x+y+1)^2} dx\, dy + \int _0^m\int _0^m \Big(\frac{Cs}{m} \Big)^2 \le Cs\log m +Cs^2\le \frac{CA^{2/3}m}{\log n},
\end{equation*}
where the first integral bounds the contribution coming from pairs $x,y$ with $|x-y|\le m/s$ such that $x<j_\ell<y$ for some $\ell \le s$, and the second integral bounds the contribution coming from pairs with $|x-y|\ge m/s$.
We turn to bound $|I_1|$. By Cauchy-Schwarz inequality we have that $|I_1|\le \sqrt{I_3I_4}$ where
\begin{equation}
\begin{split}
     I_3:= \int_0^m \int _0^m  \frac{(h(x)-h(y))^2}{|x-y|^{7/3}} dx\, dy  \quad \text{and}\quad I_4:=\int_0^m \int_0^m  \frac{(\bar g_i(x)-\bar g_i(y))^2}{|x-y|^{5/3}} dx\, dy.
\end{split}
\end{equation}
We bound $I_3$ similarly to $I_2$ by considering separately the contribution from pairs with $|x-y|\le m/s$ and pairs with $|x-y|\ge m/s$ 
\begin{equation*}
\begin{split}
     I_3 &\le s\int _0^{m/s} \int _0^{m/s} \frac{C}{(x+y+1)^{7/3}} dx\,dy + \Big(\frac{s}{m}\Big)^2 \int _0^m \int _0^m \frac{C}{(y-x)^{1/3}} dx \, dy\\
     &\le s\int _0^{m/s} \frac{C}{(x+1)^{4/3}} dx +\Big(\frac{s}{m}\Big)^2 \int _0^m C(m-x)^{2/3} dx\le Cs +\frac{Cs^2}{m^{1/3}} \le \frac{CA^{2/3}m^{2/3}}{\log n}.
\end{split}
\end{equation*}
Since $i\notin B_4$ we have that $I_4\le A^{2/3}m^{4/3}/\log n$ and therefore $|I_1|\le CA^{2/3}m/\log n$. This gives that $\theta _m^{\rho }(g)\le \theta _m^{\rho }(g_i)+CA^{2/3}m/\log n$. Moreover, we have that $\tilde \theta _m(g)\le \tilde \theta _m (g_i) +Cs\log m$, since the number of cells with hook length $x$  whose hook changed is at most $Csx$  (each of which will contribute $Cx^{-2}$ to $\tilde \theta _m $). Hence, using that $i\notin B_5$ we have that
\begin{equation}
    \Theta _m^{\rho }(g)\le \Theta _m^{\rho }(g_i)+\frac{CA^{2/3}m}{\log n} \le m\frac{\Theta _n^{\rho }(f)}{n} +\frac{CA^{2/3}m}{\log n} .
\end{equation}

It remains to prove that $g$ satisfies \eqref{eq:integral}. We have 
\begin{equation}
    \int _0^{m}  \frac{\bar g(x)^2}{x} dx \le 2\int _0^{m}  \frac{\bar g_i(x)^2}{x} dx +2\int _0^{m}  \frac{h(x)^2}{x} dx \le  \frac{CA^{2/3}m}{\log n},
\end{equation}
where in the last inequality we used that $i\notin B_2$ and that $h(x)\le Csx/m\le CA^{1/3}x/\sqrt{m\log n}$.
The second estimate in \eqref{eq:integral} is similar.
\end{proof}

\subsection{Concatenation}

In the following proposition we finish the proof of the induction in \eqref{eq:induction} using Lemma~\ref{lem:12}.

\begin{prop}\label{prop:1}
    We have that 
\begin{equation}
    \frac{\sigma _n^{\rho }}{n} \ge \sigma ^{\rho} -\frac{A}{\log n}
\end{equation}
\end{prop}

\begin{figure}[ht]
\centering


\def\ZoomPart{24,21,21,19,15,13,10,8,7,6,5,5,2}

\pgfmathsetmacro{\zoomcell}{0.12}

\newcommand{\DrawYoungRussian}[1]{%
    \foreach[count=\i from 1] \len in #1 {%
        \foreach \j in {1,...,\len} {%
            \draw (\j-\i,\i+\j-2) --
                  (\j-\i+1,\i+\j-1) --
                  (\j-\i,\i+\j) --
                  (\j-\i-1,\i+\j-1) -- cycle;
        }%
    }%
}

\newcommand{\DrawZoomUpperBoundary}{%
    \draw[line width=1pt]
    (-13,13) --
    (-11,15) --
    (-10,14) --
    (-7,17) --
    (-5,15) --
    (-4,16) --
    (-3,15) --
    (-2,16) --
    (-1,15) --
    (0,16) --
    (1,15) --
    (3,17) --
    (4,16) --
    (7,19) --
    (8,18) --
    (10,20) --
    (11,19) --
    (15,23) --
    (16,22) --
    (18,24) --
    (20,22) --
    (23,25) --
    (24,24);
}

\newcommand{\DrawOneCopy}{%
    \DrawYoungRussian{\ZoomPart}
    \DrawZoomUpperBoundary
    \draw[red,thick] (-13,13) -- (24,24);
}

\newcommand{\DrawRectBlock}[1]{%
    \foreach \i in {1,...,13} {%
        \foreach \j in {1,...,24} {%
            \filldraw[fill=#1!22,draw=#1!45,line width=0.2pt]
                (\j-\i,\i+\j-2) --
                (\j-\i+1,\i+\j-1) --
                (\j-\i,\i+\j) --
                (\j-\i-1,\i+\j-1) -- cycle;
        }%
    }%
}

\newcommand{\DrawRectBoundary}{%
    \draw[cyan!70,line width=0.7pt]
        (0,0) -- (24,24) -- (11,37) -- (-13,13) -- cycle;
}

\begin{tikzpicture}[line width=0.35pt]
\begin{scope}[x=\zoomcell cm,y=\zoomcell cm]

    \begin{scope}[shift={(13,-13)}]
        \DrawRectBlock{cyan}
    \end{scope}

    \begin{scope}[shift={(50,-2)}]
        \DrawRectBlock{cyan}
    \end{scope}

    \begin{scope}[shift={(26,-26)}]
        \DrawRectBlock{cyan}
    \end{scope}

    \DrawOneCopy

    \begin{scope}[shift={(37,11)}]
        \DrawOneCopy
    \end{scope}

    \begin{scope}[shift={(74,22)}]
        \DrawOneCopy
    \end{scope}

    \begin{scope}[shift={(13,-13)}]
        \DrawRectBoundary
    \end{scope}

    \begin{scope}[shift={(50,-2)}]
        \DrawRectBoundary
    \end{scope}

    \begin{scope}[shift={(26,-26)}]
        \DrawRectBoundary
    \end{scope}

    \fill (-13,13) circle[radius=2.2pt];
    \fill (24,24) circle[radius=2.2pt];
    \fill (61,35) circle[radius=2.2pt];
    \fill (98,46) circle[radius=2.2pt];

\end{scope}
\end{tikzpicture}

\caption{To prove Proposition~\ref{prop:1} we concatenate the flat almost minimizer constructed in Lemma~\ref{lem:12} to (essentially) itself many times and bound the interaction terms created by integrating and summing over bigger domains in $\theta ^{\rho }$ and $\tilde \theta $.}
\label{fig:4}
\end{figure}
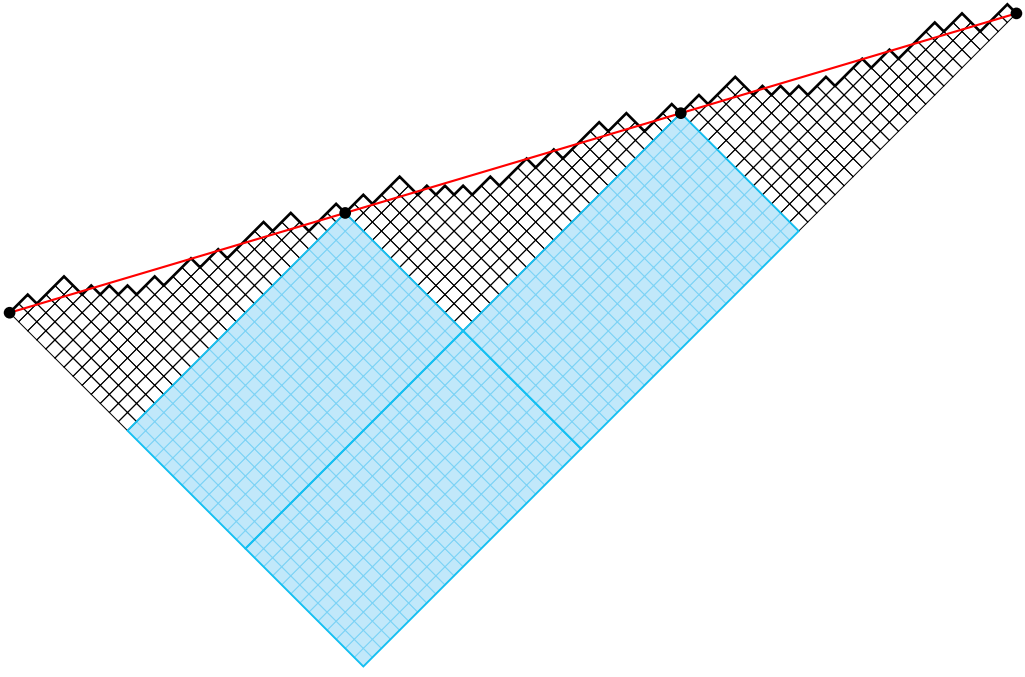

\begin{proof}
Let $\ell \ge 1$ and define $f\in \mathcal S_{\ell m}^{\rm loc}$ as follows. Suppose by induction on $0\le i\le \ell$ that $f$ was defined for all $x\in [0,im]$. If $f(im)\ge \rho im$ we let $g_{i+1}=g$ where $g$ is from Lemma~\ref{lem:12} with $a\in [\rho m-2,\rho m]$. If $f(im)< \rho im$ we let $g_{i+1}=g$ where $g$ is from Lemma~\ref{lem:12} with $a\in [\rho m,\rho m+2]$. For $x\in [im,(i+1)m]$ define $f(x):=f(im)+g_{i+1}(x-im)$. In this construction we clearly have that $|f(im)-\rho im|\le 2$ for all $i\le \ell $.

We start by bounding $\theta _{\ell m}^{\rho }(f)$. Letting $\bar f(x):=f(x)-\rho x$ we have that 
\begin{equation*}
    \theta _{\ell m}^{\rho }(f)=\sum _{i=1}^{\ell }\theta _m^{\rho }(g_i)+2\sum _{1\le i< j\le \ell } I_{i,j},\quad \text{where} \quad I_{i,j} :=\int _{(i-1)m}^{im} \int _{(j-1)m}^{jm}\Big( \frac{\bar f(y)-\bar f(x)}{y-x} \Big)^2 dx\, dy.
\end{equation*}
Next, we bound the interaction terms $I_{i,j}$. When $j=i+1$ we use the inequality $(a+b)^2\le 2a^2+2b^2$ to obtain 
\begin{equation}
\begin{split}
    I_{i,i+1} &=\int _{(i-1)m}^{im} \int _{im}^{(i+1)m}\Big( \frac{\bar f(y)-\bar f(im)}{y-x} +\frac{\bar f(im)-\bar f(x)}{y-x}  \Big)^2 dx\, dy \\
    &\le 2\int _{(i-1)m}^{im} \int _{im}^{(i+1)m}\Big( \frac{\bar f(y)-\bar f(im)}{y-x} \Big)^2 + \Big( \frac{\bar f(im)-\bar f(x)}{y-x} \Big)^2  dx\, dy \\
    &\le  2 \int _{im}^{(i+1)m} \frac{(\bar f(y)-\bar f(im))^2}{y-im} dy +2\int _{(i-1)m}^{im} \frac{(\bar f(im)-\bar f(x))^2}{im-x} dx\\
    &= 2 \int _{0}^{m} \frac{\bar g_{i+1}(y)^2}{y} dy +2\int _{0}^{m} \frac{(\bar g_i(m)-\bar g_i(x))^2}{m-x} dx \le \frac{CA^{2/3}m}{\log n},
\end{split}
\end{equation}
where in the last inequality we used the bounds in Lemma~\ref{lem:12}.

Next, using that $|\bar f(im)|\le 2$ we obtain for $j\ge i+2$
\begin{equation}
\begin{split}
I_{i,j}&= \int _{0}^{m}  \int _0^{m} \Big(\frac{\bar f((j-1)m)+\bar g_j(y)-\bar f((i-1)m)-\bar g_i(x)}{ (j-i)m +y-x } \Big)^2 dx\, dy \\
&\le  \int _{0}^{m}  \int _{0}^{m} \frac{C +C\bar g_j(y)^2+ C\bar g_i(x)^2}{m^2(j-i)^2}  dx\, dy\\
&\le \frac{C}{(j-i)^2} + \frac{C}{m(j-i)^2}\int _0^m \bar g _j(y)^2 dy +  \frac{C}{m(j-i)^2}\int _0^m \bar g _i(x)^2 dx \\
&\le \frac{C}{(j-i)^2} + \frac{C}{(j-i)^2}\int _0^m \frac{\bar g _j(y)^2}{y} dy +  \frac{C}{(j-i)^2}\int _0^m \frac{\bar g _i(x)^2}{x} dx \le \frac{CA^{2/3}m}{(j-i)^2\log n},
\end{split}
\end{equation}
where in the last inequality we used again the bound in Lemma~\ref{lem:12}.

We get that 
\begin{equation}\label{eq:theta}
\begin{split}
    \theta _{\ell m}^{\rho } (f)&\le \sum _{i=1}^{\ell }\theta _m^{\rho }(g_i)   +2\sum _{i=1}^{\ell-1} \sum _{j=i+1}^{\ell } I_{i,j}\\
    &\le  \sum _{i=1}^{\ell }\theta _m^{\rho }(g_i) +  \frac{CA^{2/3}m}{\log n} \sum _{i=1}^{\ell -1} \sum _{j=i+1}^{\ell } \frac{1}{(j-i)^2} \le  \sum _{i=1}^{\ell }\theta _m^{\rho }(g_i) + \frac{CA^{2/3}m\ell }{\log n}. 
\end{split}
\end{equation}
We turn to bound $\tilde \theta (f)$. We have that 
    \begin{equation}
        \tilde \theta _{\ell m} (f) =\sum _{i=1}^\ell \tilde \theta _m(g_i) +\sum _{1\le i <j\le \ell } S_{i,j}, 
    \end{equation}
    where $S_{i,j}$ is the contribution to $\tilde \theta _{\ell m} $ coming from cells whose hooks have their endpoints at $[(i-1)m,im]$ and $[(j-1)m,jm]$. These are the cells in one of the blue rectangles in Figure~\ref{fig:4}. We have that $S_{i,i+1}\le C\log m$ since there are at most $Ch$ cells whose hooks are of length $h$ so they will contribute $Ch\cdot h^{-2} \le Ch^{-1}$ to $S_{i,i+1}$. Summing over $h$ leads to the bound. When $j\ge i+2$ we have that $S_{i,j}\le C/(j-i)^2$ since all the hook lengths will be at least $c(j-i)m$ and there will be at most $m^2$ of them. We obtain 
\begin{equation}\label{eq:tilde}
    \begin{split}
        \tilde \theta _{\ell m} (f) &=\sum _{i=1}^\ell \tilde \theta (g_i)  +\sum _{i=1}^{\ell -1} S_{i,i+1}+\sum _{i=1}^{\ell -2}\sum _{j=i+2}^{\ell } S_{i,j} \\
        &\le \sum _{i=1}^\ell \tilde \theta (g_i) +C\ell \log m + \sum _{i=1}^{\ell -2}\sum _{j=i+2}^{\ell } \frac{C}{(j-i)^2} \le \sum _{i=1}^\ell \tilde \theta (g_i)+C\ell \log m. 
    \end{split}
    \end{equation}
    Summing \eqref{eq:theta} and \eqref{eq:tilde} and using \eqref{eq:theta62} we obtain 
    \begin{equation}\label{eq:ellm}
    \begin{split}
        \Theta _{\ell m}^{\rho }(f) &\le \sum _{i=1}^{\ell }\Theta _m^{\rho }(g_i) + \frac{CA^{2/3}m\ell }{\log n} \\
        &\le \ell \Big(m\frac{\sigma _n^{\rho }}{n}+m\frac{CA^{2/3}}{\log n}\Big) +\frac{CA^{2/3}m\ell }{\log n} 
        \le m\ell \frac{\sigma _n^{\rho }}{n} + \frac{CA^{2/3}m\ell }{\log n}. 
        \end{split}
    \end{equation}
    Hence 
    \begin{equation}
        \frac{\sigma _{m\ell }^{\rho }}{m\ell } \le \frac{\Theta _{\ell m}^{\rho }(f)}{m\ell } \le \frac{\sigma _n^{\rho }}{n } + \frac{CA^{2/3}}{\log n}. 
    \end{equation}
Taking the limit $\ell \to \infty $ we obtain 
\begin{equation}
        \sigma ^{\rho } \le \frac{\sigma _n^{\rho }}{n } + \frac{CA^{2/3}}{\log n}\le \frac{\sigma _n^{\rho }}{n } + \frac{A}{\log n}, 
    \end{equation}
    where the last inequality holds as long as $A$ is sufficiently large.
\end{proof}

Proposition~\ref{prop:1} finishes the proof of the induction in \eqref{eq:induction}. Once the induction is finished we fix $A$ and allow the constants $C,c$ to depend on $A$.

\begin{proof}[Proof of Theorem~\ref{thm:local}]
    The estimate \eqref{eq:99} follows from \eqref{eq:23} and Proposition~\ref{prop:1}.
    The fact that $\rho \mapsto \sigma ^{\rho }$ is continuous follows since $\sigma _n^{\rho }/n$ is continuous in $\rho $ (it is the minimum of polynomials in $\rho $) and $\sigma _n^{\rho }/n$ converges uniformly to $\sigma ^{\rho }$.
\end{proof}

Using the proof of Proposition~\ref{prop:1} we can prove a version of Lemma~\ref{lem:12} in which the length $m$ is of order $n$. This is stated in the following corollary.

\begin{cor}\label{cor:lem1}
    For any $n\ge 1$ there exists $m\in [n/2,n]$ such that for any $a\in \mathbb Z$ with $a\equiv m (\!\!\!\mod 2)$ and $|a-\rho m|\le 10$, there is a function $g\in \mathcal S^{\rm loc}_m$ with $g(m)=a$ such that
    \begin{enumerate}
        \item We have 
        \begin{equation}\label{eq:theta63}
            \frac{\Theta _m^{\rho }(g)}{m} \le \frac{\sigma _n^{\rho }}{n} + \frac{C}{\log n} .
        \end{equation}
        \item Letting $\bar g(x)=g(x)-\rho x$ we have
        \begin{equation}\label{eq:64}
            \int _0^{m}  \frac{\bar g(x)^2}{x} dx \le \frac{Cn}{\log n} \quad \text{and} \quad  \int _0^{m}  \frac{(\bar g (m)-\bar g(x))^2}{m-x} dx \le \frac{Cn}{\log n}  
        \end{equation}
    \end{enumerate}
\end{cor}

\begin{proof}
    The function $f$ constructed in the proof of Proposition~\ref{prop:1} with $\ell :=\lfloor n/m \rfloor$ will satisfy \eqref{eq:theta63} by \eqref{eq:ellm}. This $f$ will satisfy the first estimate in \eqref{eq:64} since 
    \begin{equation}
    \begin{split}
        \int _0^{\ell m} \frac{\bar f(x)^2}{x}dx &\le \int _0^m \frac{\bar g_1(x)^2}{x}dx +  \sum _{i=1}^{\ell -1} \int _0^m \frac{\big( \bar f(im)+\bar g_{i+1}(x) \big) ^2}{x+im}dx \\
        &\le \int _0^m \frac{\bar g_1(x)^2}{x}dx +\sum _{i=1}^{\ell -1} \int _0^m \frac{C}{im} + \frac{Cm}{im} +\frac{\bar g_{i+1} (x)^2}{x} dx  \le \frac{Cm\ell }{\log n},
    \end{split}
    \end{equation}
    where in the second inequality we expanded the square and used that $|\bar f(im)|\le C$ and that $\bar g_i(x)\le 2m$ for any $x\in [0,m]$, and in the last inequality we used \eqref{eq:integral}. The second estimate in \eqref{eq:64} is identical.

    Finally, we note the $f$ constructed in Proposition~\ref{prop:1} with $\ell :=\lfloor n/m \rfloor$ will not necessarily satisfy $f(\ell m)=a$. However, this can easily be adjusted by applying Lemma~\ref{lem:12} with a better choice of $a$ in the last two intervals. The details are omitted.
\end{proof}

\section{Global optimizers}\label{sec:global}

In this section we finish the proof of Theorem~\ref{thm:main}. This is done by integrating our local estimates from Section~\ref{sec:local} along the curve $\Omega $. For $f\in \mathcal S$ define 
$$\Theta _N (f):= \frac{1}{8}\theta _N(f) +\tilde \theta (f).$$

\subsection{Lower bound}

\begin{prop}\label{prop:lower}
    For any function $f\in \mathcal S _N$ we have that  
    \begin{equation}
        \Theta _N(f) \ge 2\sqrt{N} \int _{-1}^{1}\sigma ^{\Omega '(x)}dx -\frac{C\sqrt{N}}{\log N},
    \end{equation}
    where $\sigma ^{\rho }$ is the constant from Theorem~\ref{thm:local}.
\end{prop}

\begin{proof}
Let $f\in \mathcal S_N$ and let $n=\lfloor N^{\epsilon } \rfloor $ with $\epsilon =0.01$. Let $x_0:=-\lfloor 2\sqrt{N} -N^{1/2-\epsilon }\rfloor $ and  $x_i:=x_0+in$ and let $\ell \le CN^{1/2-\epsilon }$ be the last $i$ with $x_i \le \lfloor 2\sqrt{N} -N^{1/2-\epsilon }\rfloor$. For $1\le i\le \ell$ define $g_i\in \mathcal S^{\rm loc}_n$ by $g_i(x):=f(x_{i-1}+x)-f(x_{i-1})$ and let $\rho _i:=\Omega _N '(x_{i-1})$. We have that 
\begin{equation}\label{eq:72}
    \theta _N (f) \ge \sum  _{i=1}^{\ell } \int _{x_{i-1} } ^{x_i} \int _{x_{i-1} } ^{x_i}  \Big( \frac{f(x)-f(y)}{x-y} -\frac{\Omega _N(x)-\Omega _N(y)}{x-y}\Big) ^2 dx\, dy.
\end{equation}
Next, note that $\Omega ''(z)\le C/\sqrt{1-|z|}$ for any $z\in [-1,1]$ and therefore for any $1\le i\le \ell $ and $x,y\in [x_{i-1},x_i]$ we have
\begin{equation}\label{eq:taylor}
\begin{split}
    \Omega _N (x)-\Omega _N (y)&= \Omega _N'(y)(x-y)+O\big( (x-y)^2 N^{-1/2+\epsilon /2 } \big)\\
    &= \rho _i(x-y) +O\big( |x-y|nN^{-1/2+\epsilon /2} \big),
\end{split}
\end{equation}
where in the first equality we Taylor expanded $\Omega _N$ around $y$ and in the second equality we Taylor expanded $\Omega _N'$ around $x_{i-1}$. Dividing the last estimate by $x-y$ and substituting in \eqref{eq:72} we obtain that
\begin{equation}\label{eq:74}
\begin{split}
    \theta _N (f) &\ge \sum  _{i=1}^{\ell } \int _{0} ^{n} \int _{0} ^{n}  \Big( \frac{g_i(x)-g_i(y)}{x-y} -\rho _i +O(N^{-1/2+2\epsilon }) \Big) ^2 dx\, dy.\\
    &\ge \sum  _{i=1}^{\ell } \int _{0} ^{n} \int _{0} ^{n}  \Big( \frac{g_i(x)-g_i(y)}{x-y} -\rho _i\Big) ^2 -CN^{-1/2+2\epsilon }  dx\, dy,
\end{split}
\end{equation}
where in the last inequality we used that $\big| \frac{g_i(x)-g_i(y)}{x-y}-\rho _i\big| \le C$. Hence, 
\begin{equation}\label{eq:2.8}
    \theta _N (f) \ge  \sum  _{i=1}^\ell  \Big( \theta _n ^{\rho _i}(g_i)  -CN^{-1/2+4\epsilon }  \Big) \ge -CN^{3\epsilon }+ \sum  _{i=1}^{\ell } \theta _n^{\rho _i} (g_i).
\end{equation}
We clearly also have that $\tilde \theta (f)\ge \sum _{i=1}^\ell \tilde \theta (g_i)$ and therefore
\begin{equation}\label{eq:sumofTheta}
    \Theta _N (f) \ge -CN^{3\epsilon } + \sum _{i=1}^{\ell } \Theta _n^{\rho _i} (g_i).
\end{equation}
Using again that $\Omega ''(z)\le C/\sqrt{1-|z|}$ for $z\in [-1,1]$ we obtain that $\rho _{i+1}-\rho _i\le CN^{-1/2+3\epsilon /2}$ for all $i\le \ell $ and therefore, as in \eqref{eq:74} and \eqref{eq:2.8}, for all $\rho \in [\rho _i,\rho _{i+1}]$  we have  $|\theta _n^{\rho _i} (g_i)-\theta _n^{\rho } (g_i)| \le CN^{-1/2+4\epsilon }$. Thus, using also that $\tilde \theta (g_i)$ is independent of $\rho $, we obtain
\begin{equation}
\begin{split}
    \Theta _n^{\rho _i} (g_i)&\ge \frac{1}{n} \int _{x_{i-1}}^{x_i} \Theta _n^{\Omega_N' (x) } (g_i) dx -CN^{-1/2+4\epsilon } \\
    &\ge  \frac{1}{n} \int _{x_{i-1}}^{x_i} \sigma _n^{\Omega_N' (x) } dx -CN^{-1/2+4\epsilon }
    \ge \int _{x_{i-1}}^{x_i} \sigma ^{\Omega_N' (x) } dx -\frac{Cn}{\log N} ,
\end{split}
\end{equation}
where in the last inequality we used Theorem~\ref{thm:local}. Substituting this in \eqref{eq:sumofTheta} we obtain
\begin{equation*}
\begin{split}
    \Theta _N (f)&\ge - \frac{C\sqrt{N}}{\log N} +\int _{x_0}^{x_\ell } \sigma ^{\Omega_N' (x) } dx \\
    &\ge -\frac{C\sqrt{N}}{\log N} + \int _{-2\sqrt{N}}^{2\sqrt{N} } \sigma ^{\Omega_N' (x) }   dx = - \frac{C\sqrt{N}}{\log N}  + 2\sqrt{N}\int _{-1}^{1} \sigma ^{\Omega' (x) }  dx,
\end{split}
\end{equation*}
where in the second inequality we used Lemma~\ref{lem:11}. 
\end{proof}

\subsection{Upper bound}

\begin{prop}\label{prop:upper}
    There exists a function $f\in \mathcal S _N$ such that
    \begin{equation}
        \Theta _N(f)\le 2\sqrt{N}\int _{-1}^{1} \sigma ^{\Omega' (x) }  dx +\frac{C\sqrt{N}}{\log N}
    \end{equation}
    and such that $f(x)=|x|$ when $|x|\ge 2\sqrt{N}$.
\end{prop}

For the proof of the proposition, we will need the following lemma in which we construct a function with small $\Theta _N$ corresponding to a partition of size which is not exactly $N$. We will later adjust this function to be of size $N$ without increasing $\Theta _N$ by much. Throughout this section we fix $\epsilon =0.01$ as in the proof of Proposition~\ref{prop:lower}.

\begin{lem}\label{lem:N'}
    There exists $f\in \mathcal S$ such that 
    \begin{equation}
    \Theta _N(f)\le 2\sqrt{N}\int _{-1}^{1} \sigma ^{\Omega' (x) }  dx +\frac{C\sqrt{N}}{\log N}  
    \end{equation}
    and such that $|f(x)-\Omega _N(x)|\le N^{\epsilon }$ for all $x$ and $f(x)=|x|$ for $|x|\ge 2\sqrt{N}$.
\end{lem}

\begin{proof}
As before, we let $n=\lfloor N^{\epsilon } \rfloor $ and let $x_0:=-\lfloor 2\sqrt{N} -N^{1/2-\epsilon }\rfloor $. We let $\tilde \Omega _N$ be the maximal function in $\mathcal S$ below $\Omega _N$ and observe that $\tilde \Omega _N(x)\in [\Omega _N(x)-2,\Omega _N(x)]$ for all $x\in \mathbb R$. We define the function $f$ as follows.  For any $x\le x_0$, we let $f(x)=\tilde \Omega _N(x)$. This implies that $f(x)=|x|$ when $x<-2\sqrt{N}$. We then define a sequence of points $x_0<x_1<\cdots <x_\ell $ inductively with $x_i-x_{i-1}\in [n/2,n]$, and construct the function $f$ between these points such that $f(x_i)=\tilde \Omega _N(x_i)$. Suppose that $x_{i-1}$ was defined, for some $i\ge 1$, as well as $f(x)$ for $x\le x_{i-1}$. Let $\rho _{i}:=\Omega _N'(x_{i-1})$ and let $m_i$ be the integer $m$ from Corollary~\ref{cor:lem1} with $\rho =\rho _i$ and $n=\lfloor N^{\epsilon } \rfloor $. We let $x_i:=x_{i-1}+m_i$. Let $g_i\in \mathcal S_{m_i}^{\rm loc}$ be the function $g$ from this corollary with $a=\tilde \Omega _N (x_i)-\tilde \Omega _N(x_{i-1})$ which satisfies
$$
|a-\rho _i m_i| \le 4+ \big| \Omega _N(x_i)-\Omega _N(x_{i-1})-\rho _i m_i \big| \le 5,
$$
where the last inequality holds when $N$ is large enough by a Taylor expansion of $\Omega _N$ around $x_{i-1}$. By Corollary~\ref{cor:lem1} we have that $\Theta _{m_i}^{\rho _i} (g_i)/m_i \le \sigma_n ^{\rho _i}/n +O(1/\log N)$. We then define $f(x)=f(x_{i-1}) +g_i(x-x_{i-1})$ for any $x\in [x_{i-1},x_i]$. We stop this process when $x_i\ge 2\sqrt{N}-N^{1/2-\epsilon }$ and we let $\ell$ be the first $i$ for which that happened. Finally, define $f(x)=\tilde \Omega _N(x)$ for any $x\ge x_\ell $. Since $f(x_i)=\tilde \Omega _N(x_i)$ for all $0\le i\le \ell $ it follows that $|f(x)-\Omega _N(x)|\le n\le N^{\epsilon }$ for any $x$ and therefore for the proof of the lemma it suffices to bound $\Theta _N(f)$. As in the proof of Proposition~\ref{prop:1} we write 
\begin{equation}
    \theta _{N} (f)=\sum _{i=0}^{\ell +1}I_{i,i} +2\sum _{0\le i<j \le \ell +1 } I_{i,j} \quad \text{where} \quad   I_{i,j}:=\int _{x_{i-1}}^{x_i}\int _{x_{j-1}}^{x_j} \bigg( \frac{\bar f(x)-\bar f(y)}{x-y}  \bigg)^2 dx dy
\end{equation}
and where $x_{-1}=-\infty $ and $x_{\ell +1}=\infty $. Here $\bar f(x):=f(x)-\Omega _N(x)$. As in \eqref{eq:74} and \eqref{eq:2.8} we have that $|I_{i,i}-\theta _{m_i}^{\rho _i}(g_i)| \le CN^{-1/2+4\epsilon }$ for all $1\le i\le \ell $. Moreover, using that $|\bar f(x)|\le 2$ when $x\le x_0$ and that $\bar f(x)=0$ when $x\le -2\sqrt{N}$, we get
\begin{equation}
    I_{0,0} \le \int _{-2\sqrt{N}}^{x_0} \int _{-\infty }^{x_0} \frac{C}{(|x-y|+1)^2} dx \, dy \le C (x_0+2\sqrt{N}) \le CN^{1/2-\epsilon }.
\end{equation}
We can similarly bound $I_{\ell +1,\ell +1}$.

Next, we bound the interaction terms $I_{i,j}$ when $i\neq j$. By \eqref{eq:64} for any $1\le i\le \ell $ we have  
\begin{equation*}
    \int _{0}^{m_i} \frac{\big( g_i(x)-\rho _i x \big)^2}{x}dx \le \frac{Cn}{\log N} \quad \text{and}\quad \int _{0}^{m_i} \frac{\big( g_i(m_i)-g_i(x) -\rho_i (m_i-x) \big)^2}{m_i-x}dx\le \frac{Cn}{\log N}.
\end{equation*}
As in \eqref{eq:taylor}, we have $\Omega _N(x_{i-1}+x)-\Omega _N(x_{i-1})=\rho _ix +O(n^2N^{-1/2+\epsilon /2})$ for any $x\in [0,m_i]$ and therefore in the last bounds we may safely replace $\rho _ix$ by $\Omega _N(x_{i-1}+x)-\Omega _N(x_{i-1})$ and $\rho _i(m_i-x)$ by $\Omega _N(x_i)-\Omega _N(x_{i-1}+x)$ to obtain
\begin{equation}
    \int _{x_{i-1}}^{x_i} \frac{\big( \bar f(x)-\bar f(x_{i-1}) \big)^2}{x-x_{i-1}}dx \le \frac{Cn}{\log N} \quad \text{and}\quad \int _{x_{i-1}}^{x_i} \frac{\big( \bar f(x_i)-\bar f(x) \big)^2}{x_i-x}dx\le \frac{Cn}{\log N}.
\end{equation}
These are the only bounds that were used in the proof of Proposition~\ref{prop:1} when bounding $I_{i,j}$ and therefore, repeating the same arguments 
we obtain that $I_{i,j} \le \frac{Cn}{(j-i)^2\log N}$ for any $1\le i<j\le \ell $. Finally, using that $|\bar f(x)|\le n$ for any $x$  we obtain
\begin{equation}
    \sum _{j=1}^{\ell +1}I_{0,j} \le  \int _{x_0}^{x_{\ell +1}} \int _{-\infty }^{x_0} \frac{CN^{2\epsilon }}{(y-x+1)^2}  dx \, dy \le CN^{2\epsilon }\log N
\end{equation}
 and similarly $\sum _{i=0}^{\ell }I_{i,\ell +1} \le CN^{2\epsilon }\log N$.
 
Collecting all these bounds, we obtain as in the proof of Proposition~\ref{prop:1}, that  
\begin{equation}
    \theta _{N} (f)\le \frac{C\sqrt{N}}{\log N} +\sum _{i=1}^{\ell } \theta _{m_i}^{\rho _i}(g_i).
\end{equation}
Next, as in the proof of Proposition~\ref{prop:1}, we have that $\tilde \theta (f)\le C\ell \log N +\sum _{i=1}^{\ell} \tilde \theta _{m_i}(g_i)$ and therefore
\begin{equation}\label{eq:77}
    \Theta _{N} (f)\le \frac{C\sqrt{N}}{\log N} +\sum _{i=1}^{\ell } \Theta _{m_i}^{\rho _i}(g_i) \le \frac{C\sqrt{N}}{\log N} +\sum _{i=1}^{\ell } \frac{m_i}{n} \sigma  _{n}^{\rho _i}.
\end{equation}
The rest of the proof is similar to the lower bound. Namely, we note that for any $\rho \in [\rho _i,\rho _{i+1}]$ we have $|\sigma _{n}^{\rho _i} -\sigma _{n}^{\rho } | \le CN^{-1/2+4\epsilon }$ and therefore  
\begin{equation}
    \sigma  _{n}^{\rho _i} \le CN^{-1/2+4\epsilon } +\frac{1}{m_i} \int _{x_{i-1}}^{x_i} \sigma  _{n}^{ \Omega '_N(x)} dx \le CN^{-1/2+4\epsilon } +\frac{n}{m_i} \int _{x_{i-1}}^{x_i} \sigma  ^{ \Omega '_N(x)} dx.
\end{equation}
Substituting this in \eqref{eq:77} we obtain
\begin{equation}
    \Theta _{N} (f)\le \frac{C\sqrt{N}}{\log N} +\int _{-2\sqrt{N}} ^{2\sqrt{N}} \sigma ^{\Omega '_N(x)} dx =\frac{C\sqrt{N}}{\log N} + 2\sqrt{N}\int _{-1}^{1} \sigma ^{\Omega' (x) }  dx,
\end{equation}
as needed.
\end{proof}

\begin{proof}[Proof of Proposition~\ref{prop:upper}]
    Let $f$ be the function from Lemma~\ref{lem:N'}. Suppose without loss of generality that $N(f)<N$. The treatment when $N(f)>N$ will be identical. Observe that $N(f)$ cannot be much smaller than $N$ since
    \begin{equation}\label{eq:N'}
        N(f) =\frac{1}{2}\int _{-2\sqrt{N}}^{2\sqrt{N}}f(x)-|x|dx \ge \frac{1}{2} \int _{-2\sqrt{N}}^{2\sqrt{N}}\Omega _N(x)- N^{\epsilon }-|x|dx = N-2N^{1/2+\epsilon }.
    \end{equation}
    Next, let $R=6\lceil N^{\epsilon }\rceil $. For any integers $z\in [0,\sqrt{N}]$ and $r\ge 0$ we define the function $f_{z,r}$ as follows. When $x\notin [0,z]$ we let $f_{z,r}(x)=f(x)$. When $x\in [0,z]$ let
    \begin{equation}
        f_{z,r} (x) :=\min \big(  f(0)+x,f(x)+R, g_{z,r}(x) \big)
    \end{equation}
    where
    \begin{equation}
        g_{z,r}(x) := \begin{cases}
            f(z)+z-x ,\quad &x\in [z-r,z] \\
            \max \big( f(x), f(z)+r-(z-r-x) \big), \quad &x\in [z-r-1,z-r] \\
            \max \big( f(x), f(z)-2+z-x \big)  ,\quad &x\le z-r-1
        \end{cases}.
    \end{equation}
    This is the function that starting from $(z,f(z))$, as we go to the left from $z$ it takes $r$ steps up, then one step down and then all remaining steps up (we also maximize with $f(x)$ to make sure the down step doesn't make $g_{z,r}$ go below $f$). We clearly have that $f_{z,r}\in \mathcal S$.

\begin{figure}[ht]
\centering


\def\PartList{29,26,23,20,20,18,17,16,15,13,12,12,11,10,9,7,7,7,6,5,4,3,2,2,1,1,1,1,1,1}

\pgfmathsetmacro{\cell}{0.12}
\pgfmathsetmacro{\R}{2*sqrt(353)}   

\newcommand{\DrawYoungRussian}{
    \foreach[count=\i from 1] \len in \PartList {
        \foreach \j in {1,...,\len} {
            \draw (\j-\i,\i+\j-2) --
                  (\j-\i+1,\i+\j-1) --
                  (\j-\i,\i+\j) --
                  (\j-\i-1,\i+\j-1) -- cycle;
        }
    }
}

\newcommand{\DrawUpperBoundary}{
    \draw[line width=1pt]
    (-30,30) --
    (-29,31) --
    (-28,30) --
    (-27,29) --
    (-26,28) --
    (-25,27) --
    (-24,26) --
    (-23,25) --
    (-22,26) --
    (-21,25) --
    (-20,24) --
    (-19,25) --
    (-18,24) --
    (-17,25) --
    (-16,24) --
    (-15,25) --
    (-14,24) --
    (-13,25) --
    (-12,24) --
    (-11,25) --
    (-10,24) --
    (-9,23) --
    (-8,22) --
    (-7,23) --
    (-6,24) --
    (-5,23) --
    (-4,24) --
    (-3,23) --
    (-2,24) --
    (-1,23) --
    (0,24) --
    (1,23) --
    (2,22) --
    (3,23) --
    (4,22) --
    (5,23) --
    (6,24) --
    (7,23) --
    (8,24) --
    (9,23) --
    (10,24) --
    (11,23) --
    (12,24) --
    (13,23) --
    (14,24) --
    (15,25) --
    (16,24) --
    (17,23) --
    (18,24) --
    (19,25) --
    (20,26) --
    (21,25) --
    (22,26) --
    (23,27) --
    (24,28) --
    (25,27) --
    (26,28) --
    (27,29) --
    (28,30) --
    (29,29);
}

\begin{tikzpicture}[x=\cell cm,y=\cell cm,line width=0.35pt]

    \DrawYoungRussian

    \DrawUpperBoundary

    \draw[line width=1pt] (-30,30) -- (-\R-1.0,\R+1.0);
    \draw[line width=1pt] (29,29) -- (\R+1.0,\R+1.0);

    \draw[red,thick,samples=500,smooth,variable=\t,domain=-90:90]
        plot ({\R*sin(\t)},
              {(2*\R/pi)*((\t*pi/180)*sin(\t) + cos(\t))});

    \draw[densely dotted,line width=1pt]
    (0,24) --
    (6,30) --
    (7,29) --
    (8,30) --
    (9,29) --
    (10,30) --
    (11,29) --
    (12,30) --
    (13,29) --
    (15,31) --
    (18,28) --
    (19,29) --
    (22,26);

    \node[scale=0.85] at (11.8,25.7) {$f$};
    \node[scale=0.85] at (-18,29) {{\color{red}$\Omega_N$}};
    \node[scale=0.85] at (11.8,32) {$f_{z,r}$};

\end{tikzpicture}

\caption{The original function $f$ and its modification $f_{z,r}$ with $N(f_{z,r})=N$.}

\end{figure}
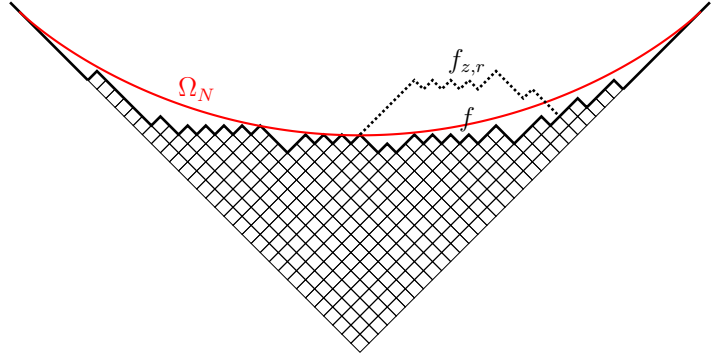

    We observe that $f(x)\le f_{z,r}(x) \le f(x)+R$ for any $x$. We have $f_{z,r}(x)=f(x)$ when $x\notin (0,z)$ and we claim that 
    \begin{equation}\label{eq:f(x)+R}
        f_{z,r}(x)=f(x)+R \quad \text{when} \quad x\in [2R,z-2R].
    \end{equation}
    To see this, suppose that $z\ge 4R$ (otherwise the statement is empty) and let $x\in [2R,z-2R]$. Using that $|f(y)-\Omega _N(y)|\le N^{\epsilon }$ for any $y$ and that $\Omega '(y)\le 1/3$ for any $y\in [0,1/2]$, we obtain
    \begin{equation*}
        f(0)+x \ge \Omega _N(0)-N^{\epsilon }+x \ge \Omega _N(x) +2x/3 -N^{\epsilon } \ge f(x)+2x/3-2N^{\epsilon } \ge f(x)+R,
    \end{equation*}
    Similarly, using also that $\Omega (y)$ is monotonically increasing when $y>0$ we obtain
    \begin{equation*}
        f(z)-2+z-x \ge \Omega _N(z)-2N^{\epsilon }+2R \ge \Omega _N(x) -2N^{\epsilon } +2R \ge f(x)-3N^{\epsilon } +2R \ge f(x)+R.
    \end{equation*}
    The equality in \eqref{eq:f(x)+R} follows from these estimates and the definition of $f_{z,r}$.

    Next, our goal is to show that for some $z,r$, we have that  $N(f_{z,r})=N$. This is done using the discrete intermediate value theorem. Observe that either $N(f_{z,r+1})=N(f_{z,r})$ or $N(f_{z,r+1})=N(f_{z,r})+1$. Moreover, $f_{z,r}=f_{z+1,0}$ as long as $r$ is sufficiently large. Hence, it suffices to prove that for $z_0:=\lfloor \sqrt{N} \rfloor $ we have that  $N(f_{z_0,0})\ge N$. Using \eqref{eq:f(x)+R} we obtain
    \begin{equation*}
    \begin{split}
        N(f_{z_0,0})=\frac{1}{2}\int _{-\infty }^{\infty } f_{z_0,0}(x)-|x|dx \ge \frac{1}{2}\int _{-\infty }^{\infty } f(x)-|x|dx +\frac{1}{2}R (z_0-4R) \\
        =N(f)+ \frac{1}{2}R (z_0-4R) \ge N,
    \end{split}
    \end{equation*}
    where in the last inequality we used \eqref{eq:N'}.

    Using the intermediate value theorem, there are $z,r$ such that $N(f_{z,r})=N$. We fix these $z,r$ from now on. For the proof of the proposition, it suffices to bound $\Theta _N (f_{z,r})$.

Let us bound $|\theta _N (f)-\theta _N (f_{z,r})|$. Define the set 
\begin{equation}
    A:=[2R,z-2R]^2\cup (\mathbb R \setminus [0,z])^2  
\end{equation}
By \eqref{eq:f(x)+R}, for $(x,y)\in A$ we have that $f(x)-f(y)=f_{z,r}(x)-f_{z,r}(y)$ and so these pairs of points will not contribute to the difference $\theta _N (f)-\theta _N (f_{z,r})$. Moreover, for any $x$ we have that $|\bar f(x)|\le N^{\epsilon }$ and $|\bar f_{z,r}(x)|\le |\bar f(x)| +R\le 9N^{\epsilon }$. Thus, we obtain
\begin{equation}
    |\theta _N (f)-\theta _N (f_{z,r})| \le \int _{\mathbb R ^2 \setminus A} \frac{CN^{2\epsilon }}{(x-y)^2+1} dx\, dy \le \int _{-\infty }^{\infty } \frac{CN^{3\epsilon }}{d_x^2+1} \, dx,
\end{equation}
where $d_x$ denotes the distance of $x$ from the set $[0,2R]\cup [z-2R,z]$. The contribution to the integral coming from $x$ with $d_x=0$ is bounded by $CRN^{3\epsilon } \le CN^{4\epsilon }$ while the contribution coming from $x$ with $d_x>0$ is bounded by $CN^{3\epsilon }$. We get that $|\theta _N (f)-\theta _N(f_{z,r})|\le CN^{4\epsilon }$. 

Next, we bound $|\tilde \theta (f)-\tilde \theta  (f_{z,r})|$. Each pair of integers  $x,y\in [-2\sqrt{N},2\sqrt{N}]$ contributes to $\tilde \theta (f)$ and $\tilde \theta  (f_{z,r})$ at most $C(x-y)^{-2}$ (each pair $(x,y)$ corresponds to the cell whose hook ends at $(x,f(x))$ and $(y,f(y))$). Pairs $(x,y)$ such that either $x,y\in [2R,z-2R]$ or $x,y\notin [0,z]$ will not contribute to the difference $\tilde \theta  (f)-\tilde \theta  (f_{z,r})$. Hence the number of pairs $(x,y)$ with $|x-y|=h$ that contribute to the difference is bounded by $CN^{\epsilon }+Ch$. This gives that $|\tilde \theta  (f)-\tilde \theta  (f_{z,r})|\le CN^{\epsilon }$. We obtain that 
\begin{equation}
    \Theta  _{N} (f_{z,r}) \le \Theta  _{N} (f) +CN^{4\epsilon } \le  2\sqrt{N}\int _{-1}^{1} \sigma ^{\Omega' (x) }  dx + \frac{C\sqrt{N}}{\log N},
\end{equation}
where in the last inequality we used the bound in Lemma~\ref{lem:N'}. 
\end{proof}

\subsection{Proof of Theorem~\ref{thm:main}}

Substituting Proposition~\ref{prop:lower} and Proposition~\ref{prop:upper} in Proposition~\ref{prop:VK} we obtain that 
\begin{equation}
    -\log \bigg( \frac{d_N^2}{(N!)^2(N/e)^{-N}}\bigg) = 2\sqrt{N} \int _{-1}^{1} \sigma ^{\Omega' (x) }  dx +O\Big( \frac{\sqrt{N}}{\log N} \Big).
\end{equation}
Indeed, for the lower bound we use that $\bar \theta _N$ is non negative and therefore the right hand side of \eqref{eq:VK} is lower bounded by $\Theta _N(f)$. For the upper bound we used that the function $f$ in Proposition~\ref{prop:upper} satisfies $f(x)=|x|$ when $|x|\ge 2\sqrt{N}$ and therefore $\bar \theta _N(f)=0$ and the right hand side of \eqref{eq:VK} is equal to $\Theta _N(f)$.

Next, by Stirling's approximation we have that $\log (N!(N/e)^{-N})= O(\log N)$ and therefore we obtain
\begin{equation}
    -\log \bigg( \frac{d_N}{\sqrt{N!}}\bigg) = \sqrt{N} \int _{-1}^{1} \sigma ^{\Omega' (x) }  dx +O\Big( \frac{\sqrt{N}}{\log N} \Big).
\end{equation}
This finishes the proof of Theorem~\ref{thm:main} with the constant $\mathfrak d$ being  $\int _{-1}^{1} \sigma ^{\Omega' (x) }  dx$.

\bibliographystyle{abbrv}
\bibliography{tree}

\end{document}